\documentclass[]{article}
\usepackage[utf8]{inputenc}
\usepackage{graphicx}

\usepackage{amsmath}
\usepackage{amsfonts}
\usepackage{caption}
\usepackage{array,multirow,makecell}
\usepackage{amsthm}
\usepackage{amssymb}

\usepackage{verbatim}
\usepackage{dsfont}
\usepackage{enumitem}

\usepackage{geometry}
\usepackage[dvipsnames]{xcolor}
\definecolor{myred}{HTML}{c20014}
\definecolor{mygreen}{HTML}{008000}

\usepackage[colorlinks=true,linkcolor=myred,citecolor=mygreen]{hyperref}

\usepackage{url}

\newtheorem{theo}{Theorem}[section]
\newtheorem{prop}{Proposition}[section]

\newtheorem{lemma}{Lemma}[section]

\theoremstyle{definition}

\newtheorem{as}{Assumption}[]

\theoremstyle{remark}
\newtheorem{rem}{Remark}[section]

\numberwithin{equation}{section}

\renewcommand{\leq}{\leqslant}
\renewcommand{\geq}{\geqslant}

\providecommand{\keywords}[1]
{
  \small	
  \textbf{{Keywords.}} #1
}

\providecommand{\AMS}[1]
{
  \small	
  \textbf{{AMS subject classifications.}} #1
}

\title{Stabilization of the wave equation through nonlinear Dirichlet actuation\thanks{
This work has been partially supported by MIAI@Grenoble Alpes (ANR-19-P3IA-0003). A preliminary version of this work, containing  the well-posedness and asymptotic stability results, has appeared in the Proceedings of the Third IFAC Conference on Modelling, Identification and Control of Nonlinear Systems, Tokyo, Japan, September 2021.
}}
\author{Nicolas Vanspranghe\thanks{Univ. Grenoble Alpes, CNRS, Grenoble INP, GIPSA-lab, 38000 Grenoble, France. Email: \nolinkurl{name.surname@gipsa-lab.fr}.}
\and Francesco Ferrante\thanks{Department of Engineering, University of Perugia, 06125 Perugia, Italy. Email: \nolinkurl{francesco.ferrante@unipg.it}.} \and Christophe Prieur\footnotemark[2]
}
\date{}

\begin{document}

\maketitle

\begin{abstract}
In this paper, we consider the problem of nonlinear (in particular, saturated) stabilization of the high-dimensional wave equation with Dirichlet boundary conditions. The wave dynamics are subject to a dissipative nonlinear velocity feedback and generate a strongly continuous semigroup of contractions on the optimal energy space $L^2(\Omega) \times H^{-1}(\Omega)$. It is first proved that any solution to the closed-loop equations converges to zero in the aforementioned topology. 
Secondly, under the condition that the feedback nonlinearity has linear growth around zero, polynomial energy decay rates are established for solutions with smooth initial data. This constitutes new Dirichlet counterparts to well-known results pertaining to nonlinear stabilization in $H^1(\Omega)\times L^2(\Omega)$ of the wave equation with Neumann boundary conditions.
\end{abstract}


\AMS{
35L05, 93C20,  93D15, 93D20.
}

\section{Introduction}
\label{sec:intro}

Let $\Omega \subset \mathbb{R}^d$ ($d \geq 2$) be a bounded domain with smooth boundary $\Gamma$. Given a relatively open nonempty subset $\Gamma_0$ of $\Gamma$, we consider the wave equation subject to non-homogeneous Dirichlet boundary conditions:
\begin{subequations}
\label{eq:pde-bc}
\begin{align}
\label{eq:pure-wave}
&\partial_{tt}u(x, t) - \Delta u(x, t) = 0 & & \mbox{in}~ \Omega \times (0, +\infty), \\
\label{eq:act-bound}
&u_{|\Gamma}(\sigma, t) = - g( U(\sigma, t)) & & \mbox{on}~ \Gamma_0 \times (0, +\infty),\\
\label{eq:hom-bound}
&u_{|\Gamma}(\sigma, t) = 0 & &\mbox{on}~ \Gamma_1 \times (0, +\infty),
\end{align}
\end{subequations}
where $\Gamma_1 \triangleq \Gamma \setminus \Gamma_0$, $U$ represents a control input, and $g$ is a real function 
fulfilling the following assumption.
\begin{as}
\label{as:g}
The scalar mapping $g$ satisfies the following properties:
\begin{enumerate}[label=(\roman*)]
\item $g$ is globally Lipschitz continuous and nondecreasing;
\item $g(s) = 0$ if and only if $s = 0$.
\end{enumerate}
\end{as}

\textbf{Background.} The general problem of this paper is the feedback stabilization of the control system \eqref{eq:pde-bc} in presence of a static pointwise nonlinearity $g$. Consider the velocity feedback
\begin{equation}
\label{eq:feedback-law}
U(\sigma, t) = - \partial_{\nu} [A^{-1}u'](\sigma, t),
\end{equation}
where $u'$ is the time derivative of $u$, 
$\partial_\nu$ denotes the outward normal derivative, and $A^{-1}$ is the inverse of the positive ``minus Laplacian with homogeneous Dirichlet boundary conditions'' operator.
The corresponding linear feedback system (i.e., when $g$ is the identity) was first introduced by Lasiecka and Triggiani in \cite{lasiecka_uniform_1987}, with initial data lying in the energy space

\begin{equation}
\mathcal{H} \triangleq L^2(\Omega) \times H^{-1}(\Omega).
\end{equation}
The choice of state space is motivated by optimal regularity results for second-order hyperbolic equations with Dirichlet boundary data in $L^2(0, T; L^2(\Gamma))$ -- see \cite{lasiecka_non_1986}. It was proved in \cite{lasiecka_uniform_1987} that the linear version of \eqref{eq:pde-bc}-\eqref{eq:feedback-law} gives rise to an exponentially stable semigroup of operators on $\mathcal{H}$ under the assumption that the whole boundary is actuated (i.e., $\Gamma = \Gamma_0$) and that $\Omega$ satisfies suitable geometrical conditions. The proof relies on the analysis of a new variable $p$ defined as
\begin{equation}
\label{eq:cv-p}
p \triangleq A^{-1} u'
\end{equation}
which is smoother and solves a wave-type equation as well. The result was later refined by the same authors in \cite{lasiecka_uniform_1992} where feedback acting only on a subset of the boundary is allowed and, most importantly, specific geometrical conditions related to the analysis of the $p$-variable by multipliers are relaxed. This was achieved by the mean of {another} change of variable operating at the level of pseudodifferential calculus. In short, after transposing problem \eqref{eq:pde-bc}-\eqref{eq:feedback-law} to the half-space via partition of unity and truncating the solution with respect to the time variable,
one defines a new variable $w$ by
\begin{equation}
\label{eq:cv-lambda}
\mathcal{F}[w](\xi, \omega; x) = \lambda(\xi, \omega) \mathcal{F}[{u}](\xi, \omega; x), \quad \xi \in \mathbb{R}^{d-1}, \quad \omega \in \mathbb{R}, \quad  x \geq 0,
\end{equation}
where $\mathcal{F}$ denotes the Fourier transform in both tangential and time variables and $\lambda$ is a carefully constructed symbol. While transformations  \eqref{eq:cv-p} and \eqref{eq:cv-lambda} are quite different in nature,  both enable computations on variables with $H^1(\Omega) \times L^2(\Omega)$-regularity.

As far as we know, there has been no attempt to extend the stability analysis of the closed-loop system \eqref{eq:pde-bc}-\eqref{eq:feedback-law} to the \emph{nonlinear} case.
Yet, one can see problem \eqref{eq:pde-bc}-\eqref{eq:feedback-law} as a natural Dirichlet counterpart to the wave equation with  nonlinear Neumann boundary dissipation
\begin{subequations}
\label{eq:neumann-problem}
\begin{align}
\label{eq:pure-wave-bis}
&\partial_{tt}u(x, t) - \Delta u(x, t) = 0 & & \mbox{in}~ \Omega \times (0, +\infty) \\
\label{eq:feedback-neumann}
&\partial_\nu u(\sigma, t) = - g(\partial_t u(\sigma, t)) & & \mbox{on}~ \Gamma_0 \times (0, +\infty),\\
&u_{|\Gamma}(\sigma, t) = 0 & &\mbox{on}~ \Gamma_1 \times (0, +\infty),
\end{align}
\end{subequations}
which, in contrast, have been extensively studied in the literature. To cite only a few, when the nonlinearity $g$ has linear growth at infinity, \emph{uniform} decay of the  $H^1(\Omega) \times L^2(\Omega)$-energy of solutions to \eqref{eq:neumann-problem} can be achieved, as in \cite{zuazua_uniform_1990} or \cite{lasiecka_uniform_1993} -- see also \cite{komornik_exact_1994} and the references therein, or more recently \cite{daoulatli_uniform_2009}.
%
In the one-dimensional settings, arguments based on Riemann invariants are available, and  the decay of the energy can be analyzed via  appropriate iterated sequences. See for instance \cite{chitour_one-dimensional_2020}, where $g$ is allowed to be a multivalued monotone mapping, or \cite{vancostenoble_optimality_2000}, where it is  proved, in particular, that exponential or polynomial uniform decay  cannot be achieved when $g$ represents a pointwise saturation mapping -- see also \cite{xu_saturated_2019} or \cite{prieur_wave_2016} for a stability analysis in the saturated case.

\textbf{Outline of the paper and contributions.}
This paper aims at bridging the gap between Neumann and Dirichlet boundary conditions as far as nonlinear boundary stabilization  is concerned. 
First, we prove that the nonlinear dynamics \eqref{eq:pde-bc}-\eqref{eq:feedback-law} generate a strongly continuous semigroup of contractions on the energy space $\mathcal{H}$ (Theorem \ref{th:had-wp}) that is globally asymptotically stable around the zero equilibrium (Theorem \ref{th:asymptotic}). The proof relies on LaSalle's invariance principle and unique continuation   for the wave  equation.

Next, having in mind the more specific problem of \emph{saturated} boundary stabilization, in Section \ref{sec:pdr}, we work under the assumption that $g$ has linear growth around zero (see Assumption \ref{as:nonlinearity} below). Then, by analogy with the  Neumann case, we focus on \emph{non-uniform} decay rates for solutions with ``smooth'' initial data. We establish a polynomial decay rate for smooth solutions (Theorem \ref{th:pdr}) that holds under standard geometrical conditions  -- see Assumption \ref{as:geometry} below, which is however always satisfied when $\Gamma_0 = \Gamma$, i.e., the whole boundary is actuated. To do so, we consider the change of variable \eqref{eq:cv-p} and we derive appropriate integral inequalities using suitable multipliers. 
 Note that the question of uniform stability when $g$ has linear growth at infinity is out of the scope of the paper -- this is discussed in  Section \ref{sec:conc} below.
Throughout the paper, one can think of the ``hard'' saturation mapping $\mathrm{sat}_S$ with threshold $S > 0$, defined by
\begin{equation}
\label{eq:def-sat}
\mathrm{sat}_S(s) \triangleq \left \{ \begin{aligned} & s && \mbox{if}~ |s|\leq S, \\ &S \frac{s}{|s|} &&\mbox{otherwise,}  \end{aligned} \right.
\end{equation}
 as a prototype nonlinearity satisfying Assumptions \ref{as:g} and \ref{as:nonlinearity}, highlighting the fact that no differentiability condition on $g$ is prescribed.


\textbf{Notation and elements from elliptic theory.} We end this section by introducing some notation and recalling useful results from elliptic theory. 

First, if $H$ is a given Hilbert space, we denote by $\|\cdot\|_H$ its norm, and its scalar product is written $(\cdot, \cdot)_H$. For $T > 0$, we denote by $W^{1, p}(0, T; H)$ the subspace of $L^p(0, T; H)$ composed of (classes of) $H$-valued functions $\phi$ such that, for some $\xi$ in $H$ and $\psi$ in $L^p(0, T; H)$, $\phi(t) = \xi + \int_0^t \psi(s) \, \mathrm{d}s $ almost everywhere in $(0, T)$.
Such a class $\phi$ is identified with its continuous representative and we say that $\phi' = \psi$ in the sense of $H$-valued distributions. Note that vector-valued integrals are intended in the sense of Bochner. Also, the space of bounded linear operators between two normed spaces $E$ and $F$ is denoted by $\mathcal{L}(E, F)$.

In this paper, all scalar functions are real-valued. The notation $\mathrm{d}x$ indicates the standard Lebesgue measure on $\mathbb{R}^d$ while $\mathrm{d}\sigma$ denotes the induced surface measure on $\Gamma$. By $H^s(\Omega)$ (resp. $H^s(\Gamma)$) we denote the $L^2(\Omega)$-based (resp. $L^2(\Gamma)$-based) Sobolev space of order $s$. The space of compactly supported and infinitely differentiable functions on $\Omega$ is written $\mathcal{C}^\infty_c(\Omega)$. We also recall that $H^1_0(\Omega)$ is defined as the closure of $\mathcal{C}_c^\infty(\Omega)$ in $H^1(\Omega)$. Furthermore, $H^{-1}(\Omega)$ is the topological dual of $H^1_0(\Omega)$.

The unbounded operator $A$ can be defined\footnote{Alternatively, $A$ can be defined as a duality mapping between $H^1_0(\Omega)$ and $H^{-1}(\Omega)$, in which case \eqref{eq:dom-A} is recovered \emph{a posteriori} by applying elliptic regularity theory.} as follows: having set
\begin{equation}
\label{eq:dom-A}
\mathcal{D}(A) \triangleq H^2(\Omega) \cap H^1_0(\Omega),
\end{equation}
we let $Au \triangleq - \Delta u \in L^2(\Omega)$ for all $u \in \mathcal{D}(A)$.
Then, $A$ is a closed strictly positive self-adjoint  operator on $L^2(\Omega)$.
Its dense domain $\mathcal{D}(A)$ is equipped with the norm $\|A\cdot\|_{L^2(\Omega)}$, which is equivalent to the norm induced by $H^2(\Omega)$. The operator $A$ possesses fractional powers $A^s$, $s \in \mathbb{R}$ -- see for instance \cite[Chapter II, Section 2.1]{temam_infinite-dimensional_2012}. Those are also strictly positive self-adjoint operators. For $s \geq 0$, $\mathcal{D}(A^s)$ are dense subsets of $L^2(\Omega)$, which we equip with the norm $\|A^s \cdot\|_{L^2(\Omega)}$. 
In particular, we have $\mathcal{D}(A^{1/2}) = H^1_0(\Omega)$, with
\begin{equation}
\|\nabla w\|^2_{L^2(\Omega)^d} = \|A^{1/2}w\|^2_{L^2(\Omega)} = \|w\|^2_{H^1_0(\Omega)} \quad \mbox{for all}~w \in H^1_0(\Omega).
\end{equation}
Then, we let $\mathcal{D}(A^{-s}) \triangleq \mathcal{D}(A^s)'$ and we can extend $A^s$ as an isomorphism between $L^2(\Omega)$ and $\mathcal{D}(A^{-s})$. Here, $H^{-1}(\Omega)$ is equipped with the scalar product  \begin{equation} (v_1, v_2)_{H^{-1}(\Omega)} \triangleq (A^{-1/2}v_1, A^{-1/2}v_2)_{L^2(\Omega)}, \end{equation}
 which induces a norm equivalent to the dual one; we also recover $\mathcal{D}(A^{-1/2}) = H^{-1}(\Omega)$. Throughout the paper, we use the following chain of continuous embeddings:
\begin{equation}
\label{eq:chain}
\mathcal{D}(A) \hookrightarrow H^1_0(\Omega) \hookrightarrow L^2(\Omega) \hookrightarrow H^{-1}(\Omega) \hookrightarrow \mathcal{D}(A^{-1}).
\end{equation}

Finally, we define the Dirichlet map $D$, which is a continuous right inverse for the trace. For any $f$ in $H^{1/2}(\Gamma)$, there exists a unique $u$ in $H^1(\Omega)$ solving $- \Delta u = 0$ and $u_{|\Gamma} = f$; and we let $Df \triangleq u$.
 The mapping  $D$ defined in this way is continuous from $H^{1/2}(\Gamma)$ to $L^2(\Omega)$
and can be extended as a continuous operator from $L^2(\Omega)$ into $L^2(\Gamma)$. We define its adjoint $D^*$  by
$(D^*u, f)_{L^2(\Gamma)} = (u, Df)_{L^2(\Omega)}$ for all $u$ in $L^2(\Omega)$ and $f$ in $L^2(\Gamma)$.
Extensions on fractional Sobolev spaces are denoted with the same symbols:
\begin{equation}
\label{eq:properties-D}
D \in \mathcal{L}(H^s(\Gamma), H^{s+1/2}(\Omega)), \quad D^* \in \mathcal{L}(H^{s}(\Omega), H^{s + 1/2}(\Gamma)) \quad \mbox{for all}~ s \in \mathbb{R}.
\end{equation}


\section{Well-posedness and asymptotic stability}
\label{sec:wp-as}

In this section, we give the operator-theoretic formulation of the evolution problem \eqref{eq:pure-wave} with initial data in $L^2(\Omega) \times H^{-1}(\Omega)$ and feedback control \eqref{eq:feedback-law}.
After that, we state  and prove the well-posedness and asymptotic stability properties of the feedback system \eqref{eq:pde-bc}-\eqref{eq:feedback-law}.

\subsection{Operator model and well-posedness}

We shall recast the closed-loop evolution equations \eqref{eq:pde-bc}-\eqref{eq:feedback-law} into a first-order abstract Cauchy problem on the energy space $\mathcal{H}$ and state well-posedness results by relying on nonlinear semigroup theory.
With a little abuse of notation, we denote by $g$ the nonlinear Lipschitz mapping on $L^2(\Gamma)$ defined by $g(f)(\sigma) \triangleq g(f(\sigma))$ for any $f$ in $L^2(\Gamma)$. We also define a projection operator $P$ on $L^2(\Gamma)$ by
$
[Pf](\sigma) = \mathds{1}_{\Gamma_0}(\sigma) f(\sigma)  ~\mbox{for any}~ f \in L^2(\Gamma).
$
From the Green formula, it follows that 
\begin{equation}
\label{eq:id-normal-derivative}
- D^*Ap = \partial_\nu p \quad \mbox{for all}~ p \in \mathcal{D}(A).
\end{equation}
Therefore, the boundary conditions  associated with the feedback law \eqref{eq:feedback-law} can be rewritten as follows:
\begin{equation}
\label{eq:func-trace}
u_{|\Gamma} =- Pg(D^*u').
\end{equation}
Next, we introduce the nonlinear operator $\mathcal{A}$ associated with the closed-loop system \eqref{eq:pde-bc}-\eqref{eq:feedback-law}. Recalling the chain of embeddings \eqref{eq:chain} and that $A$ maps $L^2(\Omega)$ onto $\mathcal{D}(A^{-1})$, we define $\mathcal{A}$ by
\begin{subequations}
\begin{align}
&\mathcal{D}(\mathcal{A}) \triangleq  \left \{ [u, v] \in \mathcal{H} :
 v \in L^2(\Omega), A[u +D P g(D^*v)] \in H^{-1}(\Omega)\right \} \\
&\mathcal{A}[u, v] \triangleq [- v, Au + ADPg(D^*v)].
\end{align}
\end{subequations}
Equivalently, $\mathcal{D}(\mathcal{A})$ is the set of all $[u, v]$ in $L^2(\Omega) \times L^2(\Omega)$ such that $u + DPg(D^*v)$ belongs to $ H^1_0(\Omega)$.

Note that the feedback law \eqref{eq:feedback-law} appears as a natural choice when (formally) differentiating the energy functional
\begin{equation}
\mathcal{E}(u, v) \triangleq \frac{1}{2}  \{ \|u\|^2_{L^2(\Omega)} + \|v \|^2_{H^{-1}(\Omega)} \}, \quad [u, v] \in \mathcal{H},
\end{equation}
along ``trajectories'' of the open-loop system \eqref{eq:pde-bc}. Indeed, this leads to the energy identity
\begin{equation}
\frac{\mathrm{d}}{\mathrm{d}t} \mathcal{E}(u, u') = \int_{\Gamma_0} g(U(t)) \partial_\nu [A^{-1}u'] \, \mathrm{d}\sigma 
\end{equation}
and since $g$ satisfies Assumption \ref{as:g}, we see that \eqref{eq:feedback-law} renders the energy $\mathcal{E}$ nonincreasing along the trajectories.

In the sequel, we employ the standard nonlinear semigroup terminology: by a \emph{strong} solution to \eqref{eq:pde-bc}-\eqref{eq:feedback-law}, we mean an absolutely continuous $\mathcal{H}$-valued function $[u, v]$ that satisfies $[u(t), v(t)] \in \mathcal{D}(\mathcal{A})$ for all $t \geq 0$ and
\begin{equation}
\frac{\mathrm{d}}{\mathrm{d}t} [u, v] + \mathcal{A}[u, v] = 0 \quad \mbox{a.e.}
\end{equation}
in the sense of strong differentiation in $\mathcal{H}$; by a \emph{generalized} solution to \eqref{eq:pde-bc}-\eqref{eq:feedback-law}, we mean a continuous $\mathcal{H}$-valued function $[u, v]$ that is, on each interval $[0, T]$, the uniform limit of some sequence of strong solutions. 
%
\begin{theo}[Hadamard well-posedness]
\label{th:had-wp}
The nonlinear operator $\mathcal{A}$ is densely defined and maximal monotone. Thus, $- \mathcal{A}$ is the infinitesimal generator of a strongly continuous semigroup $\{ \mathcal{S}_t\}$ of (nonlinear) contractions on the energy space $\mathcal{H}$. For all initial data $[u_0, v_0]$ in $\mathcal{H}$, there exists a unique generalized solution $[u, u'] \in \mathcal{C}(\mathbb{R}^+, \mathcal{H})$ to \eqref{eq:pde-bc}-\eqref{eq:feedback-law}. If $[u_0, v_0]$ belongs to $\mathcal{D}(\mathcal{A})$, then $[u, u']$ is a strong solution to \eqref{eq:pde-bc}-\eqref{eq:feedback-law}. 
Furthermore,
\begin{enumerate}[label=(\roman*)]
\item Strong solutions satisfy the inequality
\begin{equation}
\label{eq:kato-bound}
 \|\mathcal{A} [u(t), u'(t)]\|_\mathcal{H} \leq \|\mathcal{A}[u_0, v_0]\|_\mathcal{H} \quad \mbox{for all}~ t \geq 0;
\end{equation}
\item Strong solutions satisfy the energy identity
\begin{equation}
\label{eq:energy-identity}
\frac{\mathrm{d}}{\mathrm{d}t} \mathcal{E}(u, u') = - \int_{\Gamma_0} g(D^*u') D^*u' \, \mathrm{d}\sigma = \int_{\Gamma_0} g(-\partial_\nu [A^{-1}u']) \partial_\nu [A^{-1}u'] \, \mathrm{d} \sigma
\end{equation}
in the scalar distribution sense on $(0, +\infty)$.
\end{enumerate}

%

\end{theo}

\begin{rem}
If we also assume that, say, $|g(s)| \geq \alpha |s|$ for all $s \in \mathbb{R}$ and some $\alpha > 0$, the energy identity \eqref{eq:energy-identity} provides a uniform estimate of the $L^2(0, +\infty; L^2(\Gamma_0))$-norm of $\partial_\nu [A^{-1} u']$ for strong solutions. From there,  one can prove that \eqref{eq:energy-identity} holds for generalized solution as well by passing to the limit and recovering the traces $u_{|\Gamma}$ and $\partial_\nu [A^{-1} u']$ in $L^2(0, + \infty; L^2(\Gamma_0))$ -- see for instance \cite{chueshov_attractor_2002} for similar arguments in the Neumann case.
\end{rem}

\begin{proof}[Proof of Theorem \ref{th:had-wp}]
Once proven that $\mathcal{A}$ is maximal monotone, existence and uniqueness of strong and generalized solutions to \eqref{eq:pde-bc}-\eqref{eq:feedback-law}, together with the  appropriate semigroup properties, follow from Kato's theorem and standard nonlinear semigroup theory -- see, e.g.,  \cite[Chapter IV]{showalter_monotone_2013}.

\textbf{Step 1: Monotonicity.} Let $[u_1, v_1]$ and $[u_2, v_2]$ in $\mathcal{D}(\mathcal{A})$. Then,
\begin{equation} \begin{aligned}
(\mathcal{A}[u_1, v_1] - \mathcal{A}[u_2, v_2], [u_1, v_1] - [u_2, v_2])_\mathcal{H} = 
 - (v_1 - v_2, u_1 - u_2)_{L^2(\Omega)}  \\+ (A^{1/2}[u_1- u_2 +DPg(D^* v_1) -DPg(D^* v_2)], A^{-1/2}[v_1 - v_2])_{L^2(\Omega)}.
\end{aligned}
\end{equation}
Now we use that $A^{-1/2}[v_1 - v_2]$ belongs to $\mathcal{D}(A^{1/2})$ and that $A^{1/2}$ is self-adjoint to obtain 
\begin{equation}
\label{eq:AXX}
\begin{aligned}
(\mathcal{A}[u_1, v_1] - \mathcal{A}[u_2, v_2], [u_1, v_1] - [u_2, v_2])_\mathcal{H} 
&= (DPg(D^*v_1) - DPg(D^* v_2), v_1 - v_2)_{L^2(\Omega)} \\
&= (Pg(D^*v_1) - Pg(D^* v_2), D^*v_1 - D^*v_2)_{L^2(\Gamma)} \\
&= (g(D^*v_1) - g(D^* v_2), D^*v_1 - D^*v_2)_{L^2(\Gamma_0)}  \geq 0,
\end{aligned}
\end{equation}
the right-hand side being nonnegative by nondecreasingness of $g$, which proves that $\mathcal{A}$ is monotone.

\textbf{Step 2: Range condition.} Let  $\lambda > 0$ and $[f_1, f_2] \in \mathcal{H}$. To solve the equation
$
\mathcal{A}[u, v] + \lambda [u, v] = [f_1, f_2],
$
it suffices to find $v \in L^2(\Omega)$ such that
\begin{equation}
\label{eq:inconnue-v-L2}
\lambda^{-1}v + DPg(D^*v) + \lambda A^{-1}v = A^{-1}f_2 - \lambda^{-1}f_1.
\end{equation}
This is seen by substituting  $- v + \lambda u = f_1$ into the second coordinate of the equation and applying $A^{-1}$ to the result. If such an element $v \in L^2(\Omega)$ is found, then $[u, v]$ belongs to $\mathcal{D}(\mathcal{A})$ (and solves the desired equation). Indeed, we then have  $u + DPg(D^*v) + \lambda A^{-1}v = A^{-1}f_2$, which implies that $u + DPg(D^*v) \in H^1_0(\Omega)$ since $A^{-1}f_2$ and $\lambda A^{-1}v$ both belong to $H^1_0(\Omega)$.
%

We define a nonlinear operator $\Theta$ on $L^2(\Omega)$ by
$
\Theta(v) \triangleq \lambda^{-1}v + DPg(D^*v) + \lambda A^{-1}v
$ for all $v \in L^2(\Omega)$.
Then, $\Theta$ enjoys the following properties:
\begin{enumerate}[label=(\roman*)]
\item $\Theta$ maps bounded sets into bounded sets;
\item $(\Theta(v_1) - \Theta(v_2), v_1 - v_2)_{L^2(\Omega)}  \geq 0$ for all $v_1$ and $v_2$ in $L^2(\Omega)$;
\item The scalar function $t \mapsto  (\Theta(v_1 + t v_2), v_2 )_{L^2(\Omega)}$ is continuous for all $v_1$ and $v_2$ in $L^2(\Omega)$.
\end{enumerate}
Also, we have
\begin{equation}
(\Theta(v), v)_{L^2(\Omega)} \geq \lambda^{-1}\| v\|_{L^2(\Omega)} ^2 \quad \mbox{for all}~ v \in L^2(\Omega).
\end{equation}
Thus, it follows from  \cite[Lemma 2.1 and Theorem 2.1]{showalter_monotone_2013} that $\Theta$ is onto. Consequently, the equation $\mathcal{A}[u, v] + \lambda [u, v] = [f_1, f_2]$ has a solution in $\mathcal{D}(\mathcal{A})$.

\textbf{Step 3: Denseness of the domain.} Let $[u, v] \in \mathcal{H}$ and $\epsilon > 0$. Since $A^{-1}v \in H^1_0(\Omega)$ and $\mathcal{C}_c^\infty (\Omega)$ is dense in $H^1_0(\Omega)$, we can pick $\phi \in \mathcal{C}^\infty_c(\Omega)$ such that 

\begin{equation} 
\|A^{-1}v - \phi\|^2_{H^1_0(\Omega)} \leq \epsilon,
\quad \mbox{and thus}~  \|v - A \phi \|^2_{H^{-1}(\Omega)} \leq  C \epsilon
 \end{equation}
 where $C > 0$ comes from $A \in \mathcal{L}(H^1_0(\Omega), H^{-1}(\Omega))$. 
Besides, there exists $\psi \in \mathcal{C}^\infty_c(\Omega)$ such that $\| u - \psi \|^2_{L^2(\Omega)} \leq \epsilon$.
Since $\phi \in \mathcal{C}^\infty_c(\Omega) \subset \mathcal{D}(A)$, we have $A\phi \in L^2(\Omega)$ and also, using \eqref{eq:id-normal-derivative}, $g(D^*A\phi) = g(- \partial_\nu \phi) = 0$. Thus, $[\psi, A\phi] \in \mathcal{D}(\mathcal{A})$; also, we have $\| [u, v] - [\psi, A\phi]\|_\mathcal{H}^2 \leq (1 + C)\epsilon$. It is now proved that $\mathcal{D}(\mathcal{A})$ is dense in $\mathcal{H}$.

\textbf{Step 4: Energy identity.} Let $[u, v]$ be a strong solution to \eqref{eq:pde-bc}-\eqref{eq:feedback-law}. 
We recall that
$
\mathcal{E}(u, u') = \frac{1}{2} \| [u, u']\|^2_\mathcal{H}.
$
Consequently, by the chain rule, $\mathcal{E}(u, u')$ is an absolutely continuous scalar function and
\begin{equation}
\frac{\mathrm{d}}{\mathrm{d}t} \mathcal{E}(u, u') = (-\mathcal{A}[u, u'], [u, u'])_\mathcal{H} \quad \mbox{a.e.}
\end{equation}
Thus, the desired identity \eqref{eq:energy-identity} follows from \eqref{eq:AXX}.
\end{proof}

\subsection{Additional properties of the semigroup}

Now, we establish some compactness and regularity properties that are useful in the proof of the stability results presented in Subsections \ref{subsec:as} and \ref{subsec:pdr}. We start by introducing the following proposition, which enables us to prove  asymptotic stability of the feedback system \eqref{eq:pde-bc}-\eqref{eq:feedback-law} using LaSalle's invariance principle.
\begin{prop}[Compactness]
\label{prop:compact}
For any $\lambda > 0$, the (nonlinear) resolvent operator $(\mathcal{A} + \lambda \mathrm{id})^{-1}$ is well-defined on $\mathcal{H}$ and compact. In particular, for all initial data $[u_0, v_0] \in \mathcal{H}$, the (semi)trajectory $\{\mathcal{S}_t[u_0, v_0]\}_{t\geq 0}$ is relatively compact in $\mathcal{H}$.
\end{prop}
\begin{proof}
Assume for a moment that
$(\mathcal{A} + \lambda \mathrm{id})^{-1}$ is well-defined and compact for some $\lambda >0$. Then, since $\mathcal{A}(0) = 0$, relative compactness of the trajectories follows from \cite[Theorem 3]{dafermos_asymptotic_1973}.

Let $\lambda > 0$. We already know from the proof of Theorem \ref{th:had-wp} that the equation
 \begin{equation}
\label{eq:onto}
\mathcal{A}[u, v] = [f_1, f_2] \end{equation} 
has a solution in $\mathcal{D}(\mathcal{A})$ for all $[f_1, f_2] \in \mathcal{H}$.

\textbf{Step 1: Uniqueness.} Consider two solutions $[u_1, v_1]$ and $[u_2, v_2]$ to \eqref{eq:onto}. Then, we recall from \eqref{eq:inconnue-v-L2} in the proof of Theorem \ref{th:had-wp} that
\begin{equation}
\label{eq:form-uniqueness}
\lambda^{-1}[v_1 - v_2] + DP[g(D^*v_1) - g(D^*v_2)] + \lambda A^{-1}[v_1 - v_2] = 0.
\end{equation}
Taking the scalar product in $L^2(\Omega)$ of \eqref{eq:form-uniqueness} with $v_1 - v_2$ yields 
\begin{equation}
\label{eq:norme-diff}
\lambda^{-1} \|v_1 - v_2\|^2_{L^2(\Omega)} + (g(D^*v_1) - g(D^*v_2), D^*v_1 - D^*v_2)_{L^2(\Gamma_0)} + \lambda \|v_1 - v_2\|^2_{H^{-1}(\Omega)} = 0.
\end{equation}
In particular, since $g$ is nondecreasing, we infer from \eqref{eq:norme-diff} that $v_1 = v_2$; thus, $[u_1, v_1] = [u_2, v_2]$ and $(\mathcal{A} + \lambda \mathrm{id})^{-1}$ is well-defined.

\textbf{Step 2: Compactness of the resolvent operator.} In what follows, we let $[u, v] \triangleq (\mathcal{A} + \lambda \mathrm{id})^{-1}[f_1, f_2]$ and we look for estimates of $[u, v] \in \mathcal{D}(\mathcal{A})$ in stronger norms. First, as in the previous step, we obtain
\begin{equation}
\label{eq:ell-est}
\lambda^{-1} \|v \|^2_{L^2(\Omega)} + (g(D^*v), D^*v)_{L^2(\Gamma_0)} + \lambda \|v\|^2_{H^{-1}(\Omega)} =  (A^{-1/2}f_2, A^{-1/2}v)_{L^2(\Omega)} - \lambda^{-1}(f_1, v)_{L^2(\Omega)},
\end{equation}
where it is used that $A^{-1/2}$ is self-adjoint. From \eqref{eq:ell-est}, using Cauchy-Schwarz and Young inequalities with appropriate constants, we obtain the estimate
\begin{equation}
\label{eq:stronger-est-v}
 \|v \|^2_{L^2(\Omega)} \leq \frac{1}{2}\|f_2\|^2_{H^{-1}(\Omega)} +  \|f_1\|^2_{L^2(\Omega)}.
\end{equation}

The remainder of the proof relies on elliptic regularity theory and in particular \cite[Th\'eor\`eme 10.1]{lions_problemes_1961_ii}. Since $[u, v] \in \mathcal{D}(\mathcal{A})$, we know that $u + DPg(D^*v)$ belongs to $H^1_0(\Omega)$ and
\begin{equation}
\label{eq:var-u-v}
u + DPg(D^*v) = A^{-1}f_2 - \lambda A^{-1}v.
\end{equation}
Picking an arbitrary  test function $\phi$ in $\mathcal{C}^\infty_c(\Omega) \subset \mathcal{D}(A)$, taking the scalar product in $L^2(\Omega)$ of \eqref{eq:var-u-v} with $A\phi$ and using again that $D^*A\phi = - \partial_\nu \phi = 0$ leads to
$- \Delta u = f_2 - \lambda v
$ in the sense of distributions on $\Omega$.
Besides, since $- \Delta u \in H^{-1}(\Omega)$, $u_{|\Gamma}$ is well-defined in $H^{-1/2}(\Gamma)$; then, we infer from $u + DPg(D^*v) \in H^1_0(\Omega)$ that $u_{|\Gamma} = - Pg(D^*v) \in L^2(\Gamma)$. Applying the aforementioned theorem, we obtain $u \in H^{1/2}(\Omega)$ along with the estimate
\begin{equation}
\label{eq:regularity-est}
\|u\|_{H^{1/2}(\Omega)}^2 \leq C_1  \left \{  \|Pg(D^*v) \|^2_{L^2(\Gamma)} + \|f_2 - \lambda v\|^2_{H^{-1}(\Omega)}\right \}
\end{equation}
where $C_1 > 0$ is solution independent. Since $P$ and $g$ are Lipschitz continuous on $L^2(\Gamma)$, $g(0) = 0$ and $D^*$ is linear continuous from $L^2(\Omega)$ into $L^2(\Gamma)$, plugging \eqref{eq:stronger-est-v} into \eqref{eq:regularity-est} yields
\begin{equation}
\label{eq:stronger-estimate-u}
\|u\|^2_{H^{1/2}(\Omega)} \leq C_1 C_2 \|v\|^2_{L^2(\Omega)} + 2C_1\lambda^2 \|v\|^2_{H^{-1}(\Omega)} 
 + 2C_1 \|f_2\|^2_{H^{-1}(\Omega)},
\end{equation}
where $C_2 > 0$ is some other constant.

Combining \eqref{eq:stronger-est-v} and \eqref{eq:stronger-estimate-u}, we see that $(\mathcal{A} + \lambda \mathrm{id})^{-1}$ maps bounded sets of $\mathcal{H} = L^2(\Omega) \times H^{-1}(\Omega)$ into bounded sets of $H^{1/2}(\Omega) \times L^2(\Omega)$, the latter being compactly embedded into $\mathcal{H}$. Thus, the result is proved.
\end{proof}

The next proposition is meant for use in Section \ref{sec:pdr}, where we work under additional assumptions on $\Omega$; however, since it is a direct continuation of the proof of Proposition \ref{prop:compact}, we introduce it here.
\begin{prop}[Regularity]
\label{prop:reg}
 Suppose that $\overline{\Gamma_0} \cap \overline{\Gamma_1} = \emptyset$. Then, the following explicit characterization of $\mathcal{D}(\mathcal{A})$ holds:
\begin{equation}
\label{eq:carac-domain}
\mathcal{D}(\mathcal{A}) =  \left \{ [u, v] \in \mathcal{H} : v \in L^2(\Omega), u \in H^1(\Omega), u_{|\Gamma} = -  \mathds{1}_{\Gamma_0}g (D^*v)  \right \}.
\end{equation}
Thus, strong solutions $[u, u']$ take values in $H^1(\Omega) \times L^2(\Omega)$. Furthermore, there exists a constant $K > 0$ such that any strong solution to \eqref{eq:pde-bc}-\eqref{eq:feedback-law} satisfies
\begin{equation}
\label{eq:bound-H1-L2}
\|[u(t), u'(t)]\|_{H^1(\Omega) \times L^2(\Omega)} \leq K \|\mathcal{A}[u(0), u'(0)]\|_\mathcal{H} \quad \mbox{for all}~ t \geq 0.
\end{equation}
\end{prop}

\begin{rem}
\label{rem:ineq-E-A}
Since $H^1(\Omega) \times L^2(\Omega)$ is continuously embedded into $\mathcal{H}$, it follows
from \eqref{eq:bound-H1-L2} evaluated at $t = 0$ that for some constant $K' > 0$,
\begin{equation}
\label{eq:ineq-E-A}
\mathcal{E}(u_0, v_0) \leq K' \| \mathcal{A}[u_0, v_0] \|_\mathcal{H}^2 \quad \mbox{for all}~ [u_0, v_0] \in \mathcal{D}(\mathcal{A}).
\end{equation}

\end{rem}

\begin{proof}[Proof of Proposition \ref{prop:reg}]
Let $u \in H^1(\Omega)$ and $v \in L^2(\Omega)$ such that $u_{|\Gamma} = - \mathds{1}_{\Gamma_0} g(D^*v) = - Pg(D^*v)$. By trace regularity, $Pg(D^*v) \in H^{1/2}(\Gamma)$, and by \eqref{eq:properties-D}, $DPg(D^*v) \in H^1(\Omega)$.  It follows that $u + DPg(D^*v) \in H^1_0(\Omega)$, i.e., $[u, v] \in \mathcal{D}(\mathcal{A})$.


Conversely, let $[u, v] \in \mathcal{D}(\mathcal{A})$. Recalling  calculations made in Proposition \ref{prop:compact}, we already know that $[u, v]$ must satisfy 
$- \Delta u  \in H^{-1}(\Omega)$ and $u_{|\Gamma} = - Pg(D^*v)$.
Therefore,  in comparison  with the proof of Proposition \ref{prop:compact}, it suffices to show that $u_{|\Gamma}$ belongs to $H^{1/2}(\Gamma)$ instead of $L^2(\Gamma)$ and apply the elliptic regularity theorem to gain the desired extra half-unit of regularity. By virtue of \eqref{eq:properties-D}, we have $D^*v \in H^{1/2}(\Gamma)$. 

First, recall that pointwise Lipschitz nonlinearities  such as $g$ map bounded sets of $H^{1/2}(\Gamma)$ into bounded sets of $H^{1/2}(\Gamma)$. Indeed,
using the definition of Sobolev spaces on manifold by local charts and the Sobolev-Slobodeckij characterization of the fractional spaces $H^s(\mathbb{R}^{d-1})$ (see \cite{di_nezza_hitchhikers_2012}), we know that for a given $f$ in $H^{1/2}(\Gamma)$, $g(f)\in H^{1/2}(\Gamma)$ if and only if
\begin{equation}
\label{eq:frac-intr}
\iint_{\mathbb{R}^{d-1} \times \mathbb{R}^{d-1}} \frac{|\phi_i(x_1) g([f \circ \psi_i ](x_1)) - \phi_i (x_2)g([f \circ \psi_i](x_2))|^2}{\|x_1 - x_2 \|^d} \, \mathrm{d}x_1 \, \mathrm{d}x_2 < + \infty,
\end{equation}
for all suitable $(\phi_i, \psi_i)$, where the functions $\phi_i \in \mathcal{C}^\infty_c(\mathbb{R}^{d-1})$ are chosen from a partition of unity subordinate to some (finite) covering of $\Gamma$ and the functions $\psi_i$ are corresponding local representations of the surface.

The integral term in \eqref{eq:frac-intr} is finite because $g$ and $\phi_i$ are globally Lipschitz continuous; hence, $g(f) \in H^{1/2}(\Omega)$. Furthermore, taking 
  the integral term in \eqref{eq:frac-intr} 
plus some appropriate lower-order $L^2$-term
   defines a norm on $H^{1/2}(\mathbb{R}^{d-1})$ equivalent to the one given by interpolation. Thus, after coming back to functions on $\Gamma$, it follows from \eqref{eq:frac-intr} that
\begin{equation}
\label{eq:str-cont-g}
\| g(f) \|_{H^{1/2}(\Gamma)} \leq K \|f\|_{H^{1/2}(\Gamma)} \quad \mbox{for all}~ f \in H^{1/2}(\Gamma),
\end{equation}
where $K$ is some positive constant coming from the Lipschitz continuity of $g$ and norm equivalence. 

Next we have to check that $P \in \mathcal{L}(H^{1/2}(\Gamma))$. Again, this is a consequence of \eqref{eq:frac-intr}: we observe that since $\overline{\Gamma_0} \cap \overline{\Gamma_1} = \emptyset$, there exists $m > 0$ such that $\|x_1 - x_2\| > m$ whenever $(\psi_i(x_1), \psi_i(x_2)) \in [\Gamma_0 \times \Gamma_1] \cup [\Gamma_1 \times \Gamma_0]$.


Finally, combining \eqref{eq:properties-D} for $s = 1/2$, the estimate \eqref{eq:str-cont-g}, the fact that $ P \in \mathcal{L}(H^{1/2}(\Gamma))$ together with the elliptic regularity theorem, we obtain $u \in H^1(\Omega)$ and the stronger estimate
\begin{equation}
\label{eq:est-H1}
\begin{aligned}
\|u\|_{H^1(\Omega)}& \leq C \left \{ \|\Delta u \|_{H^{-1}(\Omega)} + \|v \|_{L^2(\Omega)}\right \}  \\
&= C \left \{ \|A[u + DPg(D^*v)]\|_{H^{-1}(\Omega)} + \|v \|_{L^2(\Omega)} \right \}
\leq C' \|\mathcal{A}[u, v]\|_\mathcal{H}
\end{aligned}
\end{equation}
where $C$ and $C'$ are some positive constants that do not depend on $[u, v]$. The set equality in \eqref{eq:carac-domain} is now proved and the property \eqref{eq:bound-H1-L2} readily follows from \eqref{eq:est-H1} and \eqref{eq:kato-bound}.
\end{proof}

\subsection{Asymptotic stability}
\label{subsec:as}

Next, we state the second main result of the section, which asserts that the zero equilibrium of the closed-loop system \eqref{eq:pde-bc}-\eqref{eq:feedback-law} is globally asymptotically stable.

\begin{theo}[Asymptotic stability of the closed-loop system]
\label{th:asymptotic}
Let $[u_0, v_0] \in \mathcal{H}$. Then, 
\begin{equation}
\label{eq:decay-trajectory}
\|\mathcal{S}_t[u_0, v_0]\|_\mathcal{H} \to 0 \quad \mbox{as}~ t \to + \infty. \end{equation} 
Together with the contraction property of $\{\mathcal{S}_t\}$, \eqref{eq:decay-trajectory} implies that $0$ is a globally asymptotically stable equilibrium point for the feedback system \eqref{eq:pde-bc}-\eqref{eq:feedback-law}.
\end{theo}

\begin{proof}
By the contraction property of the semigroup $\{ \mathcal{S}_t\}$ and  denseness of $\mathcal{D}(\mathcal{A})$ in $\mathcal{H}$, it suffices to prove \eqref{eq:decay-trajectory} for initial data $[u_0, v_0]$ in $\mathcal{D}(\mathcal{A})$.

To do so, we use a Lasalle-type invariance approach. Let us recall the classical line of arguments. We consider the $\omega$-limit set $\omega([u_0, v_0])$ of $[u_0, v_0]$, which can be characterized as follows: $[w_0, z_0] \in \mathcal{H}$ belongs to $\omega([u_0, v_0])$ if there exists an increasing sequence $\{t_n\} \in \mathbb{R}^\mathbb{N}$ such that $t_n \to + \infty$ and
\begin{equation}
\label{eq:carac-omega}
\mathcal{S}_{t_n}[u_0, v_0] \to [w_0, z_0] \quad \mbox{in $\mathcal{H}$ as}~ n \to + \infty.
\end{equation}
Recall that $\{\mathcal{S}_t [u_0, v_0]\}_{t \geq 0}$ is relatively compact in $\mathcal{H}$. Therefore, $\omega([u_0, v_0])$ is a nonempty (positively) invariant  compact set, and $\operatorname{dist}(\mathcal{S}_t [u_0, v_0], \omega([u_0, v_0])) \to 0$ as $t \to + \infty$ -- see \cite[Th\'eor\`eme 1.1.8]{haraux_systemes_1991}. Moreover, since $t \mapsto \|\mathcal{A}(\mathcal{S}_t[u_0, v_0])\|_\mathcal{H}$ is bounded, it follows from  \cite[Lemma 2.3]{crandall_semi-groups_1969} and \eqref{eq:carac-omega} that $\omega([u_0, v_0]) \subset \mathcal{D}(\mathcal{A})$. Besides,  since $\mathcal{E}(\mathcal{S}_t [u_0, v_0])$ is bounded and nonincreasing with respect to $t$, it must converge to some $\mathcal{E}_\infty \geq 0$ as $t$ goes to $+ \infty$.  By \eqref{eq:carac-omega} and continuity of $\mathcal{E}$, we have $\mathcal{E}(w_0, z_0) = \mathcal{E}_\infty$ for any $[w_0, z_0] \in \omega([u_0, v_0])$.

The remainder consists in proving that $\omega([u_0, v_0])$ is reduced to $\{0\}$.  Let $[w_0, z_0] \in \omega([u_0, v_0])$; we write $[w(t), w'(t)] = \mathcal{S}_t [w_0, z_0]$ and we notice that $\mathcal{E}(w(t), w'(t)) = \mathcal{E}_\infty$ for all $t \geq 0$. Furthermore, $[w, w']$ is a strong solution to \eqref{eq:pde-bc}-\eqref{eq:feedback-law} and we infer from the energy identity \eqref{eq:energy-identity} that
\begin{equation}
\label{eq:vanish-boundary}
\int_{0}^{\tau} \int_{\Gamma_0} g(- \partial_\nu [A^{-1}w']) \partial_\nu [A^{-1} w'] \, \mathrm{d} \sigma  \, \mathrm{d}t= 0 \quad \mbox{for all}~ \tau \geq 0.
\end{equation}
It follows from Assumption \ref{as:g} that $g(s)s > 0$ for every nonzero $s$.
Thus, letting $p \triangleq A^{-1}w'$, $p \in \mathcal{C}(\mathbb{R}^+, H^1_0(\Omega)) \cap L^\infty(0, +\infty; \mathcal{D}(A))$, \eqref{eq:vanish-boundary}  leads to
\begin{equation}
\label{eq:vanish-boundary-bis}
\partial_\nu p(\sigma, t) = 0 \quad \mbox{for a.e.}~ (\sigma, t) \in \Gamma_0 \times (0, + \infty).
\end{equation}
Next, we recall that $w' \in W^{1, \infty}(0, +\infty; H^{-1}(\Omega))$ and, using \eqref{eq:vanish-boundary-bis} together with the operator-theoretic formulation of \eqref{eq:pde-bc}-\eqref{eq:feedback-law}, we obtain $w'' + Aw = 0$. Hence, $p \in W^{1, \infty}(0, +\infty; H^1_0(\Omega))$ satisfies $
p' + w = 0
$, which in turn implies that $p \in W^{2, \infty}(0, +\infty; H^{-1}(\Omega))$ and solves $p'' + Ap = 0$ in $H^{-1}(\Omega)$, i.e., the standard variational formulation of the wave equation with homogeneous Dirichlet boundary conditions. In particular, $p \in \mathcal{C}^1(\mathbb{R}^+, L^2(\Omega))$ and solves the following boundary value problem:
\begin{subequations}
\begin{align}
&\partial_{tt} p - \Delta p = 0& &\mbox{in}~ \Omega \times (0, +\infty), \\
&p_{|\Gamma} = 0 & & \mbox{on}~ \Gamma \times (0, +\infty),\\
&\partial_\nu p = 0 & & \mbox{on}~ \Gamma_0 \times (0, +\infty).
\end{align}
\end{subequations}
The subset $\Gamma_0$ being relatively open in $\Gamma$, a unique continuation argument for waves yields $p = 0$  -- for instance, one can directly apply \cite[Th\'eor\`eme 2]{robbiano_theoreme_1991}. Therefore, $w' = 0$, $Aw = 0$ and finally $w = 0$, which concludes the proof.
\end{proof}

\section{Polynomial decay rates for strong solutions}
\label{sec:pdr}

This section is dedicated to the analysis of the decay rate of strong solutions under additional assumptions on the feedback nonlinearity and the geometry of the problem.
\subsection{Statement of the result and outline of the proof}
\label{subsec:pdr}

In what follows, we work under stronger assumptions that are given next.
\begin{as}
\label{as:nonlinearity}
There exist positive constants $S$, $\alpha_1$ and $\alpha_2$ such that
\begin{equation}
\label{eq:nonlinearity}
\alpha_1|s| \leq |g(s)| \leq \alpha_2 |s| \quad \mbox{for all}~ |s|\leq S.
\end{equation}
\end{as}
\begin{as}
\label{as:geometry} The domain $\Omega \subset \mathbb{R}^d$ with smooth boundary $\Gamma = \Gamma_0 \cup \Gamma_1$ satisfies the following conditions:
\begin{enumerate}
\item The boundary is such that
\begin{equation}
\overline{\Gamma_0} \cap \overline{\Gamma_1} = \emptyset;
\end{equation}
\item
There exists a point $x_0 \in \mathbb{R}^d$ such that, setting $h(x) \triangleq x - x_0$,
\begin{equation}
\label{eq:geometry-h}
h \cdot \nu \leq 0 \quad \mbox{on}~ \Gamma_1.
\end{equation}
\end{enumerate}
\end{as}

Then, we can estimate the decay rate of each strong solution.
\begin{theo}[Non-uniform polynomial decay rate]
\label{th:pdr}
Let $r \geq \max \{d-1, 2\}$.  Under Assumptions  \ref{as:nonlinearity} and \ref{as:geometry}, strong solutions $[u, u']$ to \eqref{eq:pde-bc}-\eqref{eq:feedback-law} satisfy 
\begin{equation}
\label{eq:result-pdr}
\mathcal{E}(u(t), u'(t)) \leq  C_u  t^{-\frac{2}{r-1}} \quad \mbox{for all}~ t \geq 0,
\end{equation}
where $C_u$ is a positive constant depending only on  $\mathcal{E}(u_0, v_0)$ and $\|\mathcal{A}[u_0, v_0]\|_\mathcal{H}$. 

%
\end{theo}
Theorem \ref{th:pdr} is a Dirichlet counterpart to non-uniform polynomial decay results that are well-known in the case of Neumann boundary conditions -- see, e.g., \cite[Theorem 9.10]{komornik_exact_1994}. Generally speaking, by \emph{non-uniform} we mean that the right-hand side of the inequality \eqref{eq:result-pdr} depends not only on the natural energy of the initial data $[u_0, v_0]$ (i.e., $\mathcal{E}(u_0, v_0)$) but also on higher-order terms (here, $\|\mathcal{A}[u_0, v_0]\|_\mathcal{H}$).
 As mentioned in the introduction, in the context of the wave equation with nonlinear boundary control, this type of decay is expected when the feedback nonlinearity $g$ has no linear lower bound at infinity (for instance, when $g$ is bounded and \eqref{eq:nonlinearity} cannot hold for all real numbers). Note that Theorem \ref{th:pdr} is different from \emph{uniform} polynomial decay results such as \cite[Theorem 5.1]{rao1993decay} or \cite[Theorem 2.1]{zuazua_uniform_1990}, where the corresponding feedback nonlinearities may have power growth around zero (typically leading to polynomial instead of exponential decay rate) but are still required to grow linearly at infinity (allowing uniform decay). In our problem, the non-uniform polynomial decay rate of Theorem \ref{th:pdr} is related to the lack of linear dissipation at infinity --  this will be further discussed in Section \ref{sec:conc}.

Let us introduce the following notation: if $[u, u']$ is a given solution to \eqref{eq:pde-bc}-\eqref{eq:feedback-law}, we define a (continuous) function $\mathcal{E}_u$ over $\mathbb{R}^+$ by
\begin{equation}
\mathcal{E}_u(t) \triangleq \mathcal{E}(u(t), u'(t)).
\end{equation}
Here, polynomial decay rate is obtained by applying the following classical lemma to the (nonincreasing) energy $\mathcal{E}_u$ of each solution -- see \cite[Theorem 9.1]{komornik_exact_1994} for a proof.
\begin{lemma}
\label{lem:int-ineq}
Let $E : \mathbb{R}^+ \to \mathbb{R}^+$ be a {nonincreasing} function. Assume that there exist two positive constants $\gamma$ and $T$ such that
\begin{equation}
\int_\tau^{+\infty} E^{\gamma + 1}(t) \, \mathrm{d}t \leq T E(0)^{\gamma} E(\tau) \quad \mbox{for all}~ \tau \geq 0.
\end{equation}
Then,
\begin{equation}
E(t) \leq E(0) \left ( \frac{T + \gamma t}{T + \gamma T} \right )^{- 1/ \gamma} \quad \mbox{for all}~ t \geq T.
\end{equation}
\end{lemma}
We already know from Section \ref{sec:wp-as} that $\mathcal{E}_u(t)$ converges to $0$ as $t$ goes to $+ \infty$. Our subsequent efforts focus on estimating
\begin{equation}
\label{eq:int-pow}
\int_{\tau_1}^{\tau_2} \mathcal{E}_u^{(r + 1)/2}(t) \, \mathrm{d}t \quad \mbox{for arbitrary}~ 0 \leq \tau_1 \leq \tau_2,
\end{equation}
where we recall that
\begin{equation}
\mathcal{E}_u(t) = \frac{1}{2} \{ \|u\|^2_{L^2(\Omega)} + \|u'\|^2_{H^{-1}(\Omega)} \}.
\end{equation}
As mentioned in the introduction, the proof is based on an analysis of the variable $p$ defined by

\begin{equation}
p = A^{-1}u'
\end{equation}
which solves, at least formally, the following boundary-value problem:
\begin{subequations}
\begin{align}
\label{eq:wave-p}
&\partial_{tt}p - \Delta p = - \partial_t [DPg(- \partial_\nu p)] & &\mbox{in}~ \Omega \times (0, + \infty), \\
& p_{|\Gamma} = 0 && \mbox{on}~ \Gamma \times (0, + \infty).
\end{align}
\end{subequations}
If $u'$ takes values in $L^2(\Omega)$, recalling the formula \eqref{eq:id-normal-derivative}, we have
\begin{equation}
- \partial_\nu p = - \partial_\nu [A^{-1}u'] = D^*Ap = D^*u'.
\end{equation}
To alleviate notation, in the sequel we denote by $\Phi$ the term
\begin{equation}
\Phi(t) \triangleq  DPg(D^* u'(t)).
\end{equation}
In regards to \eqref{eq:func-trace}, we see that $- \Phi$ is the harmonic extension of the trace $u_{|\Gamma}$. 
As mentioned earlier, the $p$-variable is smoother, which permits, in regards to the wave-type equation \eqref{eq:wave-p} satisfied by $p$, the use of a differential multiplier technique to obtain estimates of the integral over time of
\begin{equation}
\label{eq:energy-p}
\frac{1}{2}\int_\Omega \|\nabla p \|^2 + |p'|^2 \, \mathrm{d}x
\end{equation}
premultiplied by an appropriate power of $\mathcal{E}_u$. The quantity \eqref{eq:energy-p} is the natural energy of $[p, p']$ at the $H^1(\Omega)\times L^2(\Omega)$-level (i.e., the standard variational framework); from there, we will be able to deduce a suitable integral estimate of the energy $\mathcal{E}_u$ associated with the less regular $u$-variable.


\begin{rem}
In fact, since we want to avoid differentiating terms involving $g$ so that our results remain valid when the nonlinearity is only continuous, we will rather multiply an integrated version of \eqref{eq:wave-p}, namely the formula $u = - [p' + \Phi]$, by the time derivative of the multiplier -- see  Lemma \ref{lem:properties-p} below. In particular, $p'$ need not be continuous.
\end{rem}

\subsection{The multiplier identity}

In this subsection, we give more precise properties of the $p$-variable and  derive an expression of \eqref{eq:int-pow} in the form of a identity obtained by applying an appropriate multiplier  to \eqref{eq:wave-p}. We recall that $\mathcal{D}(A)$ is  $H^2(\Omega) \cap H^1_0(\Omega)$ equipped with the norm $\|A\cdot\|_{L^2(\Omega)}$, which is equivalent to the one induced by $H^2(\Omega)$. 
\begin{lemma}
\label{lem:properties-p}
Let $[u, u']$ be a strong solution.
The corresponding functions $p$ and $\Phi$ enjoy the regularity
\begin{equation}
p \in L^\infty(0, +\infty; \mathcal{D}(A)) \cap W^{1, \infty}(0, +\infty; H^1_0(\Omega)),  \quad \Phi \in L^\infty(0, +\infty; H^1(\Omega)),
\end{equation}
Also, the following identity holds:
\begin{equation}
\label{eq:u-p'-Phi-identity}
u = - [p' + \Phi] \in \mathcal{C}(\mathbb{R}, L^2(\Omega)).
\end{equation}
\end{lemma}


\begin{proof}
We infer from Proposition \ref{prop:reg} that
$
[u, u'] \in L^{\infty}(0, + \infty; H^1(\Omega) \times L^2(\Omega)).
$
As a consequence, since $A^{-1}$ is continuous from $L^2(\Omega)$ into $\mathcal{D}(A)$, we get
\begin{equation}
p = A^{-1} u' \in L^\infty(0, +\infty; \mathcal{D}(A)).
\end{equation}
Besides,  $u' : \mathbb{R}^+ \to H^{-1}(\Omega)$ is absolutely continuous and
\begin{equation}
\label{eq:strong-diff-H-1}
u'' + Au = -ADPg(D^*u') \quad \mbox{a.e.}
\end{equation}
Thus, applying $A^{-1}$ to \eqref{eq:strong-diff-H-1} yields
\begin{equation}
p' = A^{-1}u'' \in L^\infty(0, +\infty; H^1_0(\Omega)),
\end{equation}
and also
\begin{equation}
u = - [p' + \Phi], \quad \mbox{and}~ \Phi \in L^\infty(0, +\infty; H^1(\Omega)),
\end{equation}
which concludes the proof.
\end{proof}

Define the usual wave multiplier as follows:
\begin{equation}
\mathcal{M}p \triangleq 2h \cdot \nabla p + (d-1) p,
\end{equation}
where $h(x) = x - x_0$ as defined in Assumption \ref{as:geometry} and $d$ is the space dimension. Since $p$ satisfies a wave equation, we know that
$
\int_{\tau_1}^{\tau_2} \int_\Omega \| \nabla p \|^2 + |p'| \, \mathrm{d}x \, \mathrm{d}t
$
can be estimated by multiplying \eqref{eq:wave-p} by $\mathcal{M}p$ and integrating over $\Omega \times (\tau_1, \tau_2)$.
 Since we are looking for estimates of $\mathcal{E}_u$ at the power $(r +1)/2$, we premultiply $\mathcal{M}p$ by $\mathcal{E}_u$ at the power $(r - 1)/2$.
Thus, we shall multiply \eqref{eq:wave-p} by
\begin{equation}
\mathcal{E}_u^{(r-1)/2}(t) \mathcal{M}p(x, t).
\end{equation}
 The resulting identity is given in the next lemma.

\begin{lemma}[Multiplier identity] The following equality holds for any $0 \leq \tau_1 \leq \tau_2$:
\begin{equation}
\label{eq:multiplier-identity}
\begin{aligned}
2 \int_{\tau_1}^{\tau_2} \mathcal{E}_u^{(r+1)/2}  \, \mathrm{d}t =  \mathcal{E}_u^{(r-1)/2} \int_\Omega u \mathcal{M}p \, \mathrm{d}x \bigg |_{\tau_1}^{\tau_2} +  \int_{\tau_1}^{\tau_2} \mathcal{E}_u^{(r-1)/2} \int_\Gamma (h \cdot \nu) |\partial_\nu p|^2 \, \mathrm{d}\sigma  - \int_{\Gamma_0} (h \cdot \nu) |g(D^*u')|^2 \, \mathrm{d} \sigma \, \mathrm{d}t\\ 
- \int_{\tau_1}^{\tau_2} \mathcal{E}_u^{(r-1)/2}\int_\Omega  (d+1) \Phi u +\Phi [2h \cdot  \nabla u] \, \mathrm{d}x \, \mathrm{d}t
- \frac{(r-1)}{2}  \int_{\tau_1}^{\tau_2} \mathcal{E}_u'\mathcal{E}_u^{(r-3)/2} \int_\Omega u \mathcal{M}p \, \mathrm{d}x \, \mathrm{d}t.
\end{aligned}
\end{equation}
\end{lemma}
\begin{proof} The proof is split into four steps.

\textbf{Step 1: Integration by parts with respect to time.}
First, by linearity and continuity of $\mathcal{M}$, $\mathcal{M}p$ belongs to $W^{1, \infty}(0, +\infty; L^2(\Omega))$. On the other hand, recall that $\mathcal{E}_u$ is bounded and absolutely continuous with
\begin{equation}
\mathcal{E}_u'= - \int_{\Gamma_0} g(D^* u') D^*u' \, \mathrm{d}\sigma \quad \mbox{a.e.,} \quad \mathcal{E}'_u \in L^\infty(0, +\infty),
\end{equation}
because strong solutions are Lipschitz continuous with respect to time. Thus, $\mathcal{E}_u^{(r-1)/2}$ belongs to $W^{1, \infty}(0, +\infty)$ and $\mathcal{E}_u^{(r-1)/2} \mathcal{M}p$ belongs to $W^{1, \infty}(0, + \infty; L^2(\Omega))$ with
\begin{equation}
[\mathcal{E}_u^{(r-1)/2} \mathcal{M}p]' = \mathcal{E}_u^{(r-1)/2} \mathcal{M}p' + \frac{(r -1)}{2} \mathcal{E}_u' \mathcal{E}^{(r-3)/2}_u \mathcal{M}p \quad \mbox{a.e.}
\end{equation}
Now, it follows from \eqref{eq:u-p'-Phi-identity} that $p' + \Phi$ belongs to $W^{1, \infty}(0, +\infty; L^2(\Omega))$ and
\begin{equation}
\label{eq:p'-Phi-Ap}
-  [p' + \Phi]' = u' = Ap.
\end{equation}
Let $0 \leq \tau_1 \leq \tau_2$. Recall that, since $p \in \mathcal{D}(A)$, $Ap = -\Delta p \in L^2(\Omega)$. Thus, taking the scalar product of \eqref{eq:p'-Phi-Ap} with $\mathcal{E}_u^{(r-1)/2} \mathcal{M}p$ in $L^2(\tau_1, \tau_2; L^2(\Omega))$ and using the integration by parts formula in $W^{1, 2}(\tau_1, \tau_2; L^2(\Omega))$ leads to
\begin{equation}
\label{eq:identity-time-ipp}
\begin{aligned}
- \int_{\tau_1}^{\tau_2} \mathcal{E}_u^{(r-1)/2} \int_\Omega \Delta p \mathcal{M}p \, \mathrm{d}x \, \mathrm{d}t = \mathcal{E}_u^{(r-1)/2} \int_\Omega u \mathcal{M}p \, \mathrm{d}x \bigg |_{\tau_1}^{\tau_2} - \int_{\tau_1}^{\tau_2} \mathcal{E}_u^{(r-1)/2} \int_\Omega u \mathcal{M}p' \, \mathrm{d}x \, \mathrm{d}t \\
- \frac{(r-1)}{2} \int_{\tau_1}^{\tau_2} \mathcal{E}_u' \mathcal{E}_u^{(r-3)/2} \int_\Omega u \mathcal{M}p \, \mathrm{d}x \, \mathrm{d}t.
\end{aligned}
\end{equation}

\textbf{Step 2: Multiplier technique for the wave equation.} In what follows, we apply classical vector calculus identities to recover the $H^1_0(\Omega) \times L^2(\Omega)$-energy of $[p, p']$.
Since $p$ takes values in $H^2(\Omega)$, Rellich's identity yields
\begin{equation}
\label{eq:rellich}
\int_\Omega \Delta p [2 h\cdot \nabla p] \, \mathrm{d}x = (d - 2) \int_\Omega \| \nabla p \|^2 \, \mathrm{d}x + \int_\Gamma \partial_\nu p [2h\cdot \nabla p] \, \mathrm{d}\sigma - \int_\Gamma (h\cdot \nu) \|\nabla p\|^2 \, \mathrm{d} \sigma.
\end{equation}
Furthermore, $p_{|\Gamma} = 0$; thus
\begin{equation}
\label{eq:zero-tang}
\nabla p = (\partial_\nu p) \nu \quad \mbox{on}~ \Gamma.
\end{equation}
Combining \eqref{eq:rellich} and \eqref{eq:zero-tang}, we obtain
\begin{equation}
\label{eq:rellich-bis}
\int_\Omega \Delta p [2 h\cdot \nabla p] = (d - 2) \int_\Omega \| \nabla p\|^2 \, \mathrm{d}x + \int_\Gamma (h \cdot \nu) |\partial_\nu p|^2 \, \mathrm{d}\sigma.
\end{equation}
On the other hand,
\begin{equation}
\label{eq:ipp-Dp-p}
\int_\Omega \Delta p (d-1)p \, \mathrm{d}x = - (d-1) \int_\Omega \| \nabla p \|^2 \, \mathrm{d}x,
\end{equation}
where we use again that $p$ vanishes on the boundary. Summing \eqref{eq:rellich-bis} and \eqref{eq:ipp-Dp-p}  yields
\begin{equation}
\label{eq:u-p'-Phi}
\int_\Omega \Delta p \mathcal{M}p \, \mathrm{d}x = - \int_\Omega \|\nabla p\|^2 \, \mathrm{d}x + \int_\Gamma (h \cdot \nu) |\partial_\nu p|^2 \, \mathrm{d}x.
\end{equation}
Coming back to \eqref{eq:identity-time-ipp} and recalling \eqref{eq:u-p'-Phi-identity}, let us write
\begin{equation}
\label{eq:id-uMp'}
\begin{aligned}
\int_\Omega u \mathcal{M}p' &= - \int_\Omega p' \mathcal{M}p' \, \mathrm{d}x - \int_\Omega \Phi \mathcal{M}p' \mathrm{d}x  \\
&= -  \int_\Omega  p' [2 h\cdot \nabla p'] + (d - 1) |p'|^2 \, \mathrm{d}x   - \int_\Omega \Phi \mathcal{M}p' \, \mathrm{d}x \\
&= \int_\Omega |p'|^2 \, \mathrm{d}x   - \int_\Omega \Phi \mathcal{M}p' \, \mathrm{d}x,
\end{aligned}
\end{equation}
where we use the identity
\begin{equation}
\label{eq:id-div}
\int_\Omega \phi [2h \cdot \nabla \phi]  \, \mathrm{d}x= \int_\Gamma (h \cdot \nu) (\phi_{|\Gamma})^2 \, \mathrm{d} \sigma - \int_\Omega (\phi)^2 \operatorname{div}h \, \mathrm{d}x \quad \mbox{for all}~ \phi \in H^1(\Omega)
\end{equation}
together with $\operatorname{div} h = d$ and $p_{|\Gamma} = 0$.
Therefore, combining \eqref{eq:identity-time-ipp} with \eqref{eq:u-p'-Phi} and \eqref{eq:id-uMp'} leads to
\begin{equation}
\label{eq:conc-step2}
\begin{aligned}
\int_{\tau_1}^{\tau_2} \mathcal{E}_u^{(r-1)/2} \int_\Omega \|\nabla p\|^2 + |p'|^2 \, \mathrm{d}x \, \mathrm{d}t =  \mathcal{E}_u^{(r-1)/2} \int_\Omega u \mathcal{M}p \, \mathrm{d}x \bigg |_{\tau_1}^{\tau_2} + \int_{\tau_1}^{\tau_2} \mathcal{E}_u^{(r-1)/2} \int_\Omega \Phi \mathcal{M}p' \, \mathrm{d}x \, \mathrm{d}t \\
+ \int_{\tau_1}^{\tau_2} \mathcal{E}_u^{(r-1)/2} \int_\Gamma (h\cdot \nu) |\partial_\nu p|^2 \, \mathrm{d}\sigma \, \mathrm{d}t - \frac{(r-1)}{2} \int_{\tau_1}^{\tau_2} \mathcal{E}_u' \mathcal{E}_u^{(r-3)/2} \int_\Omega u \mathcal{M}p \, \mathrm{d}x \, \mathrm{d}t.
\end{aligned}
\end{equation}

\textbf{Step 3: Additional terms.} Here, we put the terms involving $\Phi$ into a form suitable for further estimation. It follows from \eqref{eq:u-p'-Phi-identity} that
\begin{equation}
\int_\Omega \Phi \mathcal{M}p' \, \mathrm{d}x = - \int_\Omega \Phi \mathcal{M}u \, \mathrm{d}x - \int_\Omega \Phi \mathcal{M} \Phi \, \mathrm{d}x.
\end{equation}
Applying \eqref{eq:id-div} to $\Phi$, similarly to \eqref{eq:id-uMp'}, we obtain
\begin{equation}
\label{eq:mult-phi}
\begin{aligned}
- \int_\Omega \Phi \mathcal{M} \Phi \, \mathrm{d}x &=   \int_\Omega  |\Phi|^2 \, \mathrm{d}x  - \int_\Gamma (h\cdot \nu) (\Phi_{|\Gamma})^2 \\
&=  \int_\Omega |\Phi|^2 \, \mathrm{d}x - \int_{\Gamma_0} (h \cdot \nu) |g(D^*u')|^2 \, \mathrm{d}\sigma,
\end{aligned}
\end{equation}
where we use that, by definition, 
\begin{equation}
\label{eq:trace-phi}
\Phi_{|\Gamma} = g(D^*Ap) = g(D^*u') ~\mbox{on}~\Gamma_0, \quad \Phi_{|\Gamma} = 0 ~\mbox{on}~ \Gamma_1.
\end{equation}
On the other hand, we recall that
\begin{equation}
\label{eq:phi-mult-u}
- \int_\Omega \Phi \mathcal{M} u \, \mathrm{d}x = - \int_\Omega \Phi [2 h\cdot \nabla u] + (d- 1) \Phi u  \, \mathrm{d}x.
\end{equation} 
Plugging \eqref{eq:mult-phi} and \eqref{eq:phi-mult-u} into \eqref{eq:conc-step2} 
leads to
\begin{equation}
\label{eq:conc-step3}
\begin{aligned}
\int_{\tau_1}^{\tau_2} \mathcal{E}_u^{(r-1)/2} \int_\Omega \|\nabla p\|^2 + |p'|^2 \, \mathrm{d}x \, \mathrm{d}t =  \mathcal{E}_u^{(r-1)/2} \int_\Omega u \mathcal{M}p \, \mathrm{d}x \bigg |_{\tau_1}^{\tau_2} - \frac{(r-1)}{2} \int_{\tau_1}^{\tau_2} \mathcal{E}_u' \mathcal{E}_u^{(r-3)/2} \int_\Omega u \mathcal{M}p \, \mathrm{d}x \, \mathrm{d}t \\
+ \int_{\tau_1}^{\tau_2} \mathcal{E}_u^{(r-1)/2} \int_\Gamma (h\cdot \nu) |\partial_\nu p|^2 \, \mathrm{d}\sigma - \int_{\Gamma_0} (h \cdot \nu) |g(D^*Ap)|^2 \, \mathrm{d} \sigma  \, - \int_\Omega \Phi [2h \cdot \nabla u] + (d-1)\Phi u - |\Phi|^2  \, \mathrm{d}x \, \mathrm{d}t.
\end{aligned} 
\end{equation}


\textbf{Step 4: Conclusion.} We finish the proof by rewriting \eqref{eq:conc-step3} as an estimate of $\int_{\tau_1}^{\tau_2} \mathcal{E}_u^{(r+1)/2} \, \mathrm{d}t$. This is done as follows. First, by definition of the $p$-variable,
\begin{equation}
\label{eq:u'-H-1}
\int_\Omega \|\nabla p\|^2 \, \mathrm{d}x  = \|A^{1/2}p\|^2_{L^2(\Omega)} = \|A^{-1/2}u'\|^2_{L^2(\Omega)} = \|u'\|^2_{H^{-1}(\Omega)}.
\end{equation}
On the other hand, it immediately follows from $u = -[p' + \Phi]$ that
\begin{equation}
\label{eq:p'-L2}
\int_\Omega |p'|^2 \, \mathrm{d}x = \int_\Omega |u|^2 \, \mathrm{d}x + \int_\Omega 2 \Phi u + |\Phi|^2 \, \mathrm{d}x  \quad \mbox{a.e.}
\end{equation}
Summing \eqref{eq:u'-H-1} and \eqref{eq:p'-L2} yields
\begin{equation}
\label{eq:comp-ener-u-p}
\int_\Omega \|\nabla p\|^2 + |p'|^2 \, \mathrm{d}x = 2 \mathcal{E}_u +  \int_\Omega 2 \Phi u + |\Phi|^2 \, \mathrm{d}x  \quad \mbox{a.e.}
\end{equation}
Plugging \eqref{eq:comp-ener-u-p} into \eqref{eq:conc-step3}, we get the desired identity.
\end{proof}
\begin{rem}[Disconnected boundary]
The assumption $\overline{\Gamma_0} \cap \overline{\Gamma_1} = \emptyset$ allows us to work with $H^1(\Omega)\times L^2(\Omega)$-valued strong solution to the Dirichlet problem \eqref{eq:pde-bc}-\eqref{eq:feedback-law}. The corresponding regularity for the Neumann problem \eqref{eq:neumann-problem} would be $H^2(\Omega) \times H^1(\Omega)$.
Assume for a moment that $h \cdot \nu > 0$ on $\Gamma_0$. In \cite{xu_saturated_2019,zuazua_uniform_1990}, a slightly modified version of the Neumann problem \eqref{eq:neumann-problem} is considered, namely $\partial_{tt}u - \Delta u = 0$ in $\Omega$, $\partial_\nu u = - (h \cdot \nu) g(\partial_t u)$ on $\Gamma_0$ and $u_{|\Gamma} = 0$ on $\Gamma_1$. The authors are able to relax the hypothesis $\overline{\Gamma_0} \cap \overline{\Gamma_1} = \emptyset$
 by relying on a Rellich-like \emph{inequality} that is valid in space dimension $d\leq 3$ and reads as follows:
\begin{equation}
\label{eq:rellich-inequality}
\int_\Omega \Delta \phi [2h \cdot \nabla \phi] \, \mathrm{d} x \leq (d - 2) \int_\Omega \| \nabla \phi \|^2 \, \mathrm{d}x
+ \int_\Gamma \partial_\nu \phi [2h \cdot \nabla \phi] \, \mathrm{d} \sigma -
 \int_\Gamma (h \cdot \nu)  \| \nabla \phi \|^2 \, \mathrm{d}\sigma,
\end{equation}
holding for any $\phi$ satisfying $\phi \in H^1(\Omega)$, $\Delta \phi \in L^2(\Omega)$, $\phi_{|\Gamma} = 0$ on $\Gamma_1$ and
$(h \cdot \nu)^{-1} \partial_\nu \phi  \in L^2(\Gamma_0)$. The reader is referred to \cite{grisvard1987controlabilite,komornik1990direct} for the proof of \eqref{eq:rellich-inequality}. The main benefit in that case is that the multiplier analysis can be carried out even when $\overline{\Gamma_0} \cap \overline{\Gamma_1} \not= \emptyset$, in which case strong solutions may have weaker regularity.
%


\end{rem}

\subsection{Estimates of the right-hand side and conclusion}

 Our goal in this subsection is to establish an integral inequality in the form of
\begin{equation}
\label{eq:goal}
(1 - \mu) \int_{\tau_1}^{\tau_2} \mathcal{E}_u^{(r+1)/2} \, \mathrm{d}t \leq K_u \{ \mathcal{E}_u(\tau_1) + \mathcal{E}_u (\tau_2) \} \quad \mbox{for all}~ 0 \leq \tau_1 \leq \tau_2,
\end{equation}
where $K_u$ is a constant that may depend on the initial data and $\mu$ is a sufficiently small constant that may depend on $u$ as well. Bearing in mind the full statement of Theorem \ref{th:pdr}, we aim at finding such constants that depend on $\|\mathcal{A}[u_0, v_0]\|_\mathcal{H}$ and $ \mathcal{E}(u_0, v_0) = \mathcal{E}_u(0)$ only.

 Assuming that \eqref{eq:goal} holds, we let $\tau_2$ go to $+ \infty$ to obtain
\begin{equation}
\int_{\tau}^{+\infty} \mathcal{E}_u^{(r + 1)/2} \, \mathrm{d}t \leq \frac{K_u}{1 - \mu}\mathcal{E}_u(\tau) \quad \mbox{for all}~\tau \geq 0.
\end{equation}
Then, Theorem \ref{th:pdr} follows readily from Lemma \ref{lem:int-ineq} if we choose 
\begin{equation}
\label{eq:constants-pdr}
\gamma = \frac{(r-1)}{2} \quad \mbox{and}~ T = \frac{K_u}{1 - \mu}\mathcal{E}_u(0)^{-(r-1)/2}.
\end{equation}
To prove \eqref{eq:goal},  we shall examine each term in the multiplier identity \eqref{eq:multiplier-identity} and derive estimates in terms of
\begin{itemize}
\item Either directly $\mathcal{E}_u(\tau_1)$ or $\mathcal{E}_u(\tau_2)$;
\item Or the boundary dissipation term  $ \int_{\tau_1}^{\tau_2}\int_{\Gamma_0} g(D^*u') D^*u' \, \mathrm{d} \sigma$ which is nonnegative and can be integrated, since
\begin{equation}
- \mathcal{E}_u'=  \int_{\Gamma_0} g(D^*u')D^*u' \, \mathrm{d}\sigma = - \int_{\Gamma_0} g(-\partial_\nu p) \partial_\nu p \, \mathrm{d}\sigma;
\end{equation}
\item And also $\int_{\tau_1}^{\tau_2} \mathcal{E}^{(r+1)/2}_u \, \mathrm{d}t$ premultiplied by small  $\mu$ so that it can be absorbed in the left-hand side.
\end{itemize}

\begin{rem} In what follows, we shall denote by $K$, $K'$, etc. generic constants that \emph{do not} depend on the initial data.
\end{rem}
We can write estimates in terms of $\int_{\tau_1}^{\tau_2} \int_{\Gamma_0} |g(D^*u')|^2 \, \mathrm{d}\sigma \, \mathrm{d}t$ instead of $\int_{\tau_1}^{\tau_2} \int_{\Gamma_0} g(D^*u')D^*u' \, \mathrm{d}\sigma \, \mathrm{d}t$ since
\begin{equation}
0 \leq \int_{\Gamma_0} |g(D^*u')|^2 \, \mathrm{d}\sigma \leq K \int_{\Gamma_0} |g(D^*u')D^*u'| \, \mathrm{d}\sigma =  K \int_{\Gamma_0} g(D^*u')D^*u' \, \mathrm{d}\sigma
\end{equation}
by Lipschitz continuity and nonincreasingness of $g$, together with $g(0) = 0$.

%

Moreover, recalling \eqref{eq:geometry-h} in Assumption \ref{as:geometry} and looking at the sign of each term, we observe that $|\partial_\nu p|$ in \eqref{eq:multiplier-identity} need not be estimated on the uncontrolled boundary $\Gamma_1$. 

That being said, let us start by estimating the term involving $2 h\cdot \nabla u$ in \eqref{eq:multiplier-identity}. This is done in the following lemma.
\begin{lemma}
\label{lem:hard-term}
Suppose that $r \geq 2$. Then, there exists a positive constant $K$ such that
\begin{multline}
\label{eq:hard-term}
\left |  \int_{\tau_1}^{\tau_2} \mathcal{E}_u^{(r-1)/2}  \int_\Omega \Phi [2 h\cdot \nabla u] \, \mathrm{d}x \, \mathrm{d}t  \right |  \\ \leq 
K \|\mathcal{A}[u_0, v_0]\|^{1/2}_\mathcal{H} \left \{  \int_{\tau_1}^{\tau_2} \frac{1}{\mu} \int_{\Gamma_0} |g(D^*u')|^2 \, \mathrm{d}\sigma +
\mu \mathcal{E}_u^{(r -2)/2}(0) \mathcal{E}_u ^{(r+1)/2} \, \mathrm{d}t \right \}
\end{multline}
 for all $\tau_2 \geq \tau_1 \geq 0$ and $\mu > 0$.
\end{lemma}

\begin{proof}
We start by writing
\begin{equation}
(\Phi, 2h\cdot \nabla u)_{L^2(\Omega)} = (DPg(D^*u'), 2h\cdot \nabla u)_{L^2(\Omega)} = (g(D^*u'), D^*[2h\cdot\nabla u])_{L^2(\Gamma_0)}.
\end{equation}
Thus, applying the Cauchy-Schwarz inequality yields
\begin{equation}
\label{eq:cs}
\left |\int_\Omega \Phi [2h\cdot \nabla u] \, \mathrm{d}x \right | \leq  \|g(D^*u')\|_{L^2(\Gamma_0)} \|D^*[2h \cdot \nabla u ]\|_{L^2(\Gamma_0)} \leq  \|g(D^*u')\|_{L^2(\Gamma_0)} \|D^*[2h \cdot \nabla u ]\|_{L^2(\Gamma)}.
\end{equation}
Next, recall from \eqref{eq:properties-D} that
\begin{equation}
\label{eq:cont-D}
D^* \in \mathcal{L}(H^{-1/2}(\Omega), L^2(\Gamma)).
\end{equation}
Therefore, it follows from \eqref{eq:cs} that
\begin{equation}
\label{eq:CS}
\left |\int_\Omega \Phi  [2h\cdot \nabla u] \, \mathrm{d}x \right | \leq K  \|g(D^*Ap)\|_{L^2(\Gamma_0)} \|2h \cdot \nabla u\|_{H^{-1/2}(\Omega)}.
\end{equation}
Linear interpolation between the Sobolev spaces $L^2(\Omega)$ and $H^{-1}(\Omega)$ leads to
\begin{equation}
\|2 h\cdot \nabla u\|_{H^{-1/2}(\Omega)} \leq K \|2 h\cdot \nabla u\|^{1/2}_{L^2(\Omega)} \|2 h\cdot \nabla u\|^{1/2}_{H^{-1}(\Omega)}
\end{equation}
First, by Proposition \ref{prop:reg},
\begin{equation}
\label{eq:bound-graph}
\|2 h\cdot \nabla u\|^{1/2}_{L^2(\Omega)} \leq  K \|h \|_{L^\infty(\Omega)^d}^{1/2}\|\nabla u \|_{L^2(\Omega)^d}^{1/2} \leq  K' \| \mathcal{A}[u_0, v_0] \|_\mathcal{H}^{1/2}.
\end{equation}
Besides,  since $2h\cdot \nabla u$ belongs to $L^2(\Omega)$,
we have
\begin{equation}
\langle 2h\cdot \nabla u, w \rangle_{H^{-1}(\Omega), H^1_0(\Omega)} = (2h\cdot \nabla u, w)_{L^2(\Omega)} \quad \mbox{for all}~ w \in H^1_0(\Omega).
\end{equation}
Let us write
\begin{equation}
\begin{aligned}
\int_\Omega [2 h\cdot \nabla u] w \, \mathrm{d}x &= 2 \sum_{i=1}^{d} \int_\Omega h_i \partial_{i} u w \, \mathrm{d}x \\
&= - 2 \sum_{i=1}^{d} \int_\Omega u [w \partial_{i} h_i + h_i \partial_{i} w] \, \mathrm{d}x +\sum_{i=1}^d \int_\Gamma \nu_i h_i u w \, \mathrm{d}\sigma \\
&= - 2 \int_\Omega u [w \operatorname{div} h +  h\cdot \nabla w] \, \mathrm{d}x,
\end{aligned}
\end{equation}
where $h_i$ (resp. $\nu_i$) denotes the $i$-th coordinate of the vector field $h$ (resp. the outward normal vector $\nu$). Recall that $h \in \mathcal{C}^1(\overline{\Omega})$. Then, using the Poincaré inequality on $w$,  we obtain that for some $K > 0$, 
\begin{equation}
\label{eq:ineq-scal}
\left |(2 h\cdot \nabla u, w)_{L^2(\Omega)} \right| \leq K \| u\|_{L^2(\Omega)} \| w \|_{H^1_0(\Omega)}.
\end{equation}
By norm equivalence between $H^{-1}(\Omega)$ and $H^1_0(\Omega)'$,
 we infer from \eqref{eq:ineq-scal} that
\begin{equation}
\label{eq:cont-H-1-L2}
\|2 h\cdot \nabla u\|_{H^{-1}(\Omega)} \leq K \|u\|_{L^2(\Omega)}.
\end{equation}
Coming back to \eqref{eq:CS}, combining \eqref{eq:bound-graph} and \eqref{eq:cont-H-1-L2} yields
\begin{equation}
\begin{aligned}
\left |\int_\Omega Dg(D^*u') [2h\cdot \nabla u] \, \mathrm{d}x \right |  & \leq K \|\mathcal{A}[u_0, v_0]\|^{1/2}_\mathcal{H} \|g(D^*u')\|_{L^2(\Gamma_0)} \|u\|^{1/2}_{L^2(\Omega)}
\\  &\leq K'  \|\mathcal{A}[u_0, v_0]\|^{1/2}_\mathcal{H} \|g(D^*u')\|_{L^2(\Gamma_0)} \mathcal{E}_u^{1/4}.
\end{aligned}
\end{equation}
Therefore, since $\mathcal{E}_u \geq 0$, we have
\begin{equation}
\left | \mathcal{E}_u^{(r-1)/2}  \int_\Omega Dg(D^*u') 2 h\cdot \nabla u \, \mathrm{d}x  \right | \leq K  \|\mathcal{A}[u_0, v_0]\|^{1/2}_\mathcal{H} \|g(D^*u')\|_{L^2(\Gamma_0)} \mathcal{E}_u^{(r-1)/2 + 1/4}.
\end{equation}
Applying Young's inequality with a parameter $\mu > 0$, we obtain
\begin{equation}
\label{eq:first-young}
\|g(D^*u')\|_{L^2(\Gamma_0)}\mathcal{E}^{(r-1)/2 +1/4}_u \leq \frac{1}{2\mu} \|g(D^*u')\|_{L^2(\Gamma_0)}^{2} + \frac{\mu}{2} \mathcal{E}_u^{r - 1/2}.
\end{equation}
It is assumed that $r \geq 2$. Thus, letting $\eta \triangleq  (r - 1/2) - (r +1)/2 \geq 0$, by nonincreasingness of $\mathcal{E}_u$, we have
\begin{equation}
\label{eq:power-exp}
\mathcal{E}_u^{r-1/2} = \mathcal{E}_u^\eta \mathcal{E}^{(r+1)/2} \leq \mathcal{E}_u^\eta(0) \mathcal{E}_u^{(r+1)/2}.
\end{equation}
Plugging \eqref{eq:power-exp} into \eqref{eq:first-young} and integrating over $(\tau_1, \tau_2)$ yields the desired result.
\end{proof}

Next, we deal with the term involving $|\partial_\nu p| = |D^*u'|$ on the controlled boundary $\Gamma_0$. Here, the arguments are very similar to those employed in the case of saturated Neumann feedback -- see \cite[Theorem 9.10]{komornik_exact_1994} or \cite[Theorem 3]{xu_saturated_2019}.
\begin{lemma} 
\label{lemma:sat}
Suppose that $r \geq d - 1$. Then, there exists $ K > 0$  and  $\eta \in (0, 1)$ such that
\begin{equation}
\begin{aligned}
\left | \int_{\tau_1}^{\tau_2}  \mathcal{E}_u^{(r-1)/2} \int_{\Gamma_0} |\partial_\nu p|^2 \, \mathrm{d}\sigma \, \mathrm{d}t  \right |  \leq
 K  \| \mathcal{A}[u_0, v_0]\|_\mathcal{H}^{2 - \eta}    \mu^{ 1/\eta} \int_{\tau_1}^{\tau_2} \int_{\Gamma_0} g(D^*u') D^*u' \, \mathrm{d}\sigma \, \mathrm{d}t  \\ + K  \| \mathcal{A}[u_0, v_0]\|_\mathcal{H}^{2 - \eta} \mu^{-1/(1 - \eta)} \int_{\tau_1}^{\tau_2} \mathcal{E}_u^{(r+1)/2} \, \mathrm{d}t 
+  K \int_{\tau_1}^{\tau_2} \int_{\Gamma_0} |g(D^*u')|^2 \, \mathrm{d}\sigma \, \mathrm{d}t  
\end{aligned}
\end{equation}
for all $0 \leq \tau_1 \leq \tau_2$ and $\mu > 0$.
\end{lemma}
\begin{proof}
For each $t \geq 0$, we set
\begin{equation}
\Gamma^{0}_t \triangleq \left \{ \sigma \in \Gamma_0 :  |\partial_\nu p(\sigma, t)| \leq S \right \} \quad \mbox{and}~ \Gamma^1_t  \triangleq \Gamma_0 \setminus \Gamma^{0}_t.
\end{equation}
where we recall that the constant $S$ is introduced in Assumption \ref{as:nonlinearity}.
 Then, 
\begin{equation}
\label{eq:decoupage}
\left | \int_{\Gamma_0} (h \cdot \nu)|\partial_\nu p|^2 \, \mathrm{d}\sigma \right| \leq K 
 \int_{\Gamma^0_t} |\partial_\nu p|^2 \, \mathrm{d}\sigma + K \int_{\Gamma_t^{1}} |\partial_\nu p|^2 \, \mathrm{d}\sigma.
\end{equation}
Using \eqref{eq:nonlinearity} in Assumption \ref{as:nonlinearity}, we estimate the first term in \eqref{eq:decoupage} as follows:
\begin{equation}
\label{eq:sat-first}
 \int_{\Gamma^0_t} |\partial_\nu p|^2 \, \mathrm{d}\sigma \leq \alpha_1^{-2} \int_{\Gamma^0_t} |g(- \partial_\nu p)|^2 \, \mathrm{d}\sigma \leq \alpha_1^{-2} \int_{\Gamma_0} |g(- \partial_\nu p)|^2 \, \mathrm{d} \sigma = \alpha_1^{-2} \int_{\Gamma_0} |g(D^*u')|^2 \, \mathrm{d}\sigma.
\end{equation}
Let us examine the second term in \eqref{eq:decoupage}. Setting a  parameter $\eta \in (0, 1)$ to be tuned later on, we have
\begin{equation}
\label{eq:times-div}
\int_{\Gamma^1_t} | \partial_\nu p|^2  \, \mathrm{d}\sigma = \int_{\Gamma^1_t} |D^*u'|^2 \, \mathrm{d}\sigma = \int_{\Gamma^1_t} |D^*u'|^{2 - \eta} \frac{|g(D^*u') D^*u'|^\eta}{|g(D^*u')|^\eta} \, \mathrm{d}\sigma.
\end{equation}
Equation \eqref{eq:times-div} makes sense since $|g(D^*u')| \geq \min \{ g(S), -g(-S)\} > 0$ on $\Gamma^{1}_t$. In fact, we have
\begin{equation}
\label{eq:insert-power}
\begin{aligned}
\int_{\Gamma^1_t} |D^*u'|^2 \, \mathrm{d}\sigma &\leq \min \{ g(S), -g(-S)\}^{-\eta} \int_{\Gamma^1_t} |D^*u'|^{2 - \eta} |g(D^*u')D^*u'|^\eta \, \mathrm{d}\sigma \\
&\leq K \int_{\Gamma_0}  |D^*u'|^{2 - \eta} |g(D^*u')D^*u'|^\eta \, \mathrm{d}\sigma.
\end{aligned}
\end{equation}
Using H\"older's inequality with conjugates $1/\eta$ and $1/(1 - \eta)$, we infer from \eqref{eq:insert-power} that
\begin{equation}
\label{eq:holder}
\int_{\Gamma^1_t} |D^*u'|^2 \, \mathrm{d}\sigma \leq K \left ( \int_{\Gamma_0} |D^*u'|^{\frac{2 - \eta}{1 - \eta}} \right )^{1- \eta} \left ( \int_{\Gamma_0} g(D^*u') D^*u' \, \mathrm{d}\sigma \right )^{\eta}.
\end{equation}
Now, $[u, u']$ being a strong solution to \eqref{eq:pde-bc}-\eqref{eq:feedback-law}, we recall from Proposition \ref{prop:reg} that $u'$ takes values in $L^2(\Omega)$ and $\|u'(t)\|_{L^2(\Omega)} \leq K \|\mathcal{A}[u_0, v_0]\|_\mathcal{H}$ for all $t \geq 0$. The continuity property \eqref{eq:properties-D} yields
\begin{equation}
\label{eq:bound-D}
\|D^*u'(t)\|_{H^{1/2}(\Gamma)} \leq K \|\mathcal{A}[u_0, v_0]\|_\mathcal{H} \quad \mbox{for all}~ t \geq 0.
\end{equation}
In what follows, we rely on (fractional) Sobolev inequalities -- see \cite[Theorem 6.5 and Theorem 6.9]{di_nezza_hitchhikers_2012}. First, we consider the case $d \geq 3$, where we recall that $d$ denotes the space dimension. We have the continuous embedding
\begin{equation}
\label{eq:embedding}
H^{1/2}(\Gamma_0) \hookrightarrow L^q(\Gamma_0) \quad \mbox{for all}~ q \in \left  [2, \frac{2(d - 1)}{d - 2} \right ] \triangleq I_d.
\end{equation}
Furthermore, since $r + 1 \geq d$, if we choose $\eta = 2/(r+1)$, some computations yield $(2 - \eta)/(1 - \eta) \in I_d $; hence
\begin{equation}
\label{eq:embedding-eta}
H^{1/2}(\Gamma_0) \hookrightarrow L^{\frac{2 - \eta}{1 - \eta}}(\Gamma_0).
\end{equation}
If $d= 2$, then the embedding \eqref{eq:embedding} holds in fact  for any $q \in [2, +\infty)$; therefore, \eqref{eq:embedding-eta} is valid as well. Coming back to \eqref{eq:holder}, combining \eqref{eq:embedding-eta} with \eqref{eq:bound-D} yields
\begin{equation}
\mathcal{E}_u^{(r-1)/2} \int_{\Gamma^1_t} |D^*u'|^2 \, \mathrm{d}\sigma \leq K  \| \mathcal{A}[u_0, v_0]\|_\mathcal{H}^{2 - \eta} \mathcal{E}_u^{(r-1)/2} \left ( \int_{\Gamma_0}g(D^*u') D^*u' \, \mathrm{d}\sigma \right )^{\eta}.
\end{equation}
Applying the Young inequality with conjugates $1/\eta$ and $1/( 1 - \eta)$, we get
\begin{equation}
\label{eq:embed-y}
\mathcal{E}_u^{(r-1)/2} \int_{\Gamma^1_t} |D^*u'|^2 \, \mathrm{d}\sigma \leq  K  \| \mathcal{A}[u_0, v_0]\|_\mathcal{H}^{2 - \eta} \left \{ \mu^{-\frac{1}{1 - \eta}} \mathcal{E}_u^{\frac{r - 1}{2(1 - \eta)}} + \mu^{\frac{1}{\eta}} \int_{\Gamma_0} g(D^*u') D^*u' \, \mathrm{d} \sigma \right \} \quad \mbox{for all}~ \mu > 0.
\end{equation}
Since $(r-1)/2(1  -\eta) = (r +1)/2$,
we conclude the proof by combining \eqref{eq:embed-y} and \eqref{eq:sat-first} together with \eqref{eq:decoupage}.
\end{proof}

At this point, the proof of Theorem \ref{th:pdr}  is almost complete. Estimates of the remaining terms in \eqref{eq:multiplier-identity} are given in the next lemmas. Following our remarks at the beginning of the subsection, we claim that Theorem \ref{th:pdr} is  proved once those are established.
\begin{lemma}
There exists a positive constant $K$ such that
\begin{equation}
\label{eq:integrated-term}
\left | \left [ \mathcal{E}^{(r-1)/2}_u \int_\Omega u \mathcal{M}p \, \mathrm{d}x \right ]_{\tau_1}^{\tau_2} \right | 
\leq K \mathcal{E}^{(r-1)/2}_u(0) \{ \mathcal{E}_u(\tau_1) + \mathcal{E}_u(\tau_2) \} \quad \mbox{for all}~ 0 \leq \tau_1 \leq \tau_2.
\end{equation}
\end{lemma}
\begin{proof}
Let $\tau \geq 0$. Then,
\begin{equation}
\begin{aligned}
\left |\mathcal{E}_u^{(r-1)/2}(\tau) \int_\Omega u(\tau) \mathcal{M}p(\tau) \, \mathrm{d}x \right | &\leq  \mathcal{E}^{(r-1)/2}_u(\tau) \|u(\tau)\|_{L^2(\Omega)} \| \mathcal{M}p(\tau)\|_{L^2(\Omega)} \\
&\leq K \mathcal{E}_u^{(r-1)/2}(\tau) \|u(\tau)\|_{L^2(\Omega)} \|p(\tau)\|_{H^1_0(\Omega)} \\
&\leq K \mathcal{E}^{(r-1)/2}(\tau)\|u(\tau)\|_{L^2(\Omega)} \|u'(\tau)\|_{H^{-1}(\Omega)} \\
&\leq K \mathcal{E}^{(r-1)/2}_u(0) \mathcal{E}_u(\tau),
\end{aligned}
\end{equation}
where it used that $\mathcal{E}_u$ is nonincreasing. Equation \eqref{eq:integrated-term} readily follows from the triangular inequality.
\end{proof}

\begin{lemma} There exists a positive constant $K$ such that
\begin{equation}
\left |  \int_{\tau_1}^{\tau_2} \mathcal{E}_u'\mathcal{E}_u^{(r-3)/2} \int_\Omega u \mathcal{M}p \, \mathrm{d}x \, \mathrm{d}t \right | \leq K \mathcal{E}_u^{(r-1)/2}(0) \{ \mathcal{E}_u(\tau_1) + \mathcal{E}_u(\tau_2) \} \quad \mbox{for all}~ 0 \leq \tau_1 \leq \tau_2.
\end{equation}
\end{lemma}
\begin{proof}
Again, we write
\begin{equation}
\left | \int_\Omega u \mathcal{M}p \, \mathrm{d}x \right | \leq K \mathcal{E}_u.
\end{equation}
Therefore,
\begin{equation}
\label{eq:term-derivative}
\begin{aligned}
\left |  \int_{\tau_1}^{\tau_2} \mathcal{E}_u'\mathcal{E}_u^{(r-3)/2} \int_\Omega u \mathcal{M}p \, \mathrm{d}x \, \mathrm{d}t \right | \leq K \int_{\tau_1}^{\tau_2} (- \mathcal{E}'_u) \mathcal{E}_u^{(r - 1)/2} \, \mathrm{d}t &= - K \int_{\tau_1}^{\tau_2} \left [\frac{2}{r + 1} \mathcal{E}^{(r + 1)/2} \right ]' \mathrm{d}t  \\&= \frac{2K}{r + 1} \{ \mathcal{E}^{(r+1)/2}_u(\tau_1) - \mathcal{E}^{(r + 1)/2}_u(\tau_2) \}.
\end{aligned}
\end{equation}
The desired inequality follows from the nonincreasingness of $\mathcal{E}_u$ and \eqref{eq:term-derivative}.
\end{proof}
\begin{lemma}
There exists a positive contant $K$ such that 
\begin{equation}
\left | \int_{\tau_1}^{\tau_2} \mathcal{E}_u^{(r-1)/2} \int_\Omega \Phi u \, \mathrm{d}x \, \mathrm{d}t \right   | \leq K \left \{ \frac{1}{\mu}  \int_{\tau_1}^{\tau_2} \mathcal{E}_u^{(r+1)/2} \, \mathrm{d}t + \mu \mathcal{E}_u^{(r-1)/2}(0) \int_{\tau_1}^{\tau_2} \int_{\Gamma_0} |g(D^*u')|^2 \, \mathrm{d}\sigma \, \mathrm{d}t \right \}
\end{equation}
for all $0 \leq \tau_1 \leq \tau_2$ and $\mu > 0$.
\end{lemma}
\begin{proof} 
First, using Cauchy-Schwarz and Young inequalities, we obtain
\begin{equation}
\label{eq:ez-cs}
\begin{aligned}
\left | \int_{\tau_1}^{\tau_2} \mathcal{E}_u^{(r-1)/2} \int_\Omega \Phi u \, \mathrm{d}x \, \mathrm{d}t \right |& \leq \int_{\tau_1}^{\tau_2} \mathcal{E}_u^{(r-1)/2} \|\Phi\|_{L^2(\Omega)} \|u\|_{L^2(\Omega)} \, \mathrm{d}t \\
& \leq \frac{1}{2\mu} \int_{\tau_1}^{\tau_2} \mathcal{E}_u^{(r-1)/2} \|\Phi\|^2_{L^2(\Omega)} \, \mathrm{d}t + \mu \int_{\tau_1}^{\tau_2} \mathcal{E}_u^{(r+1)/2} \, \mathrm{d}t.
\end{aligned}
\end{equation}
Next, recall that
$\Phi = DPg(D^*u')$ and that $D$ continuously maps $L^2(\Gamma)$ into $L^2(\Omega)$. Therefore,
\begin{equation}
\label{eq:cont-D-Phi}
\|\Phi \|^2_{L^2(\Omega)} \leq K \|Pg(D^*u')\|_{L^2(\Gamma)}^2 = K \|g(D^*u')\|^2_{L^2(\Gamma_0)}
\end{equation}
and we conclude the proof by plugging \eqref{eq:cont-D-Phi} into \eqref{eq:ez-cs} and using that $\mathcal{E}_u$ is nonincreasing.
\end{proof}

\section{Concluding remarks}
\label{sec:conc}

In this section, we discuss the results of our paper and give some comments and perspectives.
\begin{itemize}
\item Theorem \ref{th:pdr} deals with the decay rate of strong solutions to \eqref{eq:pde-bc}-\eqref{eq:feedback-law}, which remain bounded in a stronger norm (here, in $H^1(\Omega)\times L^2(\Omega)$). In particular, this enables the use of Sobolev embeddings to obtain estimates of the boundary term $\partial_\nu [A^{-1}u']$ in $L^\infty(0, +\infty; L^q(\Gamma))$ for some appropriate $q$. In view of the energy identity \eqref{eq:energy-identity}, and as done in Lemma \ref{lemma:sat},  we can then derive an estimate involving only the ``dissipation term'' $g(-\partial_\nu[A^{-1}u'])$ and lower-order energy terms, even though no lower bound on the nonlinearity $g$ is prescribed at infinity. Here, using the terminology of \cite{vancostenoble_optimality_2000}, the feedback is allowed to be \emph{weak}, i.e., $g(s)/s$ can go to $0$ as $|s|$ goes to infinity, as it is the case when $g$ represents a saturation mapping; then, loss of uniformity is to be expected. More precisely, coming back to the Neumann problem,  the one-dimensional version of \eqref{eq:neumann-problem} with $g$ given by \eqref{eq:def-sat} is known to possess weak solutions that decay to zero (in the natural energy space $H^1(\Omega) \times L^2(\Omega)$) slower than any exponential or polynomial, whereas strong solutions decay exponentially to zero but in a non-uniform way -- see \cite[Theorem 4.1]{vancostenoble_optimality_2000} or also \cite[Theorem 4.33]{chitour_one-dimensional_2020}. Proving a similar result in our Dirichlet case would be interesting.
\item Putting aside the matter of saturated feedback and assuming if needed that $g$ has linear growth at infinity, we see that, unfortunately, the strategy followed here is not sufficient to prove \emph{uniform} decay of solutions to \eqref{eq:pde-bc}-\eqref{eq:feedback-law}. Indeed, while estimating the term $(\Phi, 2 h\cdot \nabla u)_{L^2(\Omega)}$ as in Lemma \ref{lem:hard-term} is good enough for the purpose of proving Theorem \ref{th:pdr}, it requires, again, that solutions remain bounded in a norm stronger than that of the energy space $\mathcal{H}$. If, instead of \eqref{eq:hard-term}, one manages to prove something in the likes of
\begin{equation}
\int_0^\tau \int_\Omega \Phi [2h\cdot \nabla u] \, \mathrm{d}x \, \mathrm{d}t \leq K(\tau) \int_0^\tau |g(D^*u')|^2 \, \mathrm{d}\sigma \, \mathrm{d}t + K' \{ \mathcal{E}_u(0) + \mathcal{E}_u(\tau) \} + \epsilon \int_0^\tau \mathcal{E}_u \, \mathrm{d}t
\end{equation}
for some $\tau >0$, where $K(\tau)$ is allowed to depend on $\tau$ and $\epsilon$ can be chosen sufficiently small, then, by remarking that the multiplier identity \eqref{eq:multiplier-identity} is still valid with the time-varying weight $\mathcal{E}(u, u')^{(r-1)/2}$ replaced by the constant $1$, one could easily adapt the rest of our proof to obtain exponential uniform stability. By following the proof of \cite[Lemma 3.3]{lasiecka_uniform_1987}, we can prove such an estimate when $g$ is the identity, at least under some specific geometrical conditions; however, the argument breaks down in the nonlinear case. Therefore, as mentioned in the introduction, the problem of uniform stability is still open.
\end{itemize}

\bibliographystyle{alpha}
\bibliography{wave-dirichlet}  

\end{document}